\documentclass[12pt,centertags,oneside]{amsart}
\usepackage{amsmath,amstext,amsthm,amscd,typearea}
\usepackage{amssymb}
\usepackage[backref=page]{hyperref}

\usepackage{amscd,amsxtra,calc}
\usepackage{cmmib57}
\usepackage{url}
\usepackage[T1]{fontenc}
\usepackage{lmodern}
\usepackage{braket}
\usepackage{amscd}
\usepackage{cancel}
\usepackage{color}
\usepackage[normalem]{ulem}
\usepackage{xcolor}
\usepackage{microtype}

\usepackage[a4paper,width=16.2cm,top=3cm,bottom=3cm]{geometry}

\numberwithin{equation}{section}
\theoremstyle{plain}
\newtheorem{thm}{Theorem}[section]
\newtheorem{theorem}[thm]{Theorem}

\newtheorem{lemma}[thm]{Lemma}
\newtheorem{lem}[thm]{Lemma}

\newtheorem{cor}[thm]{Corollary}

\newtheorem{pro}[thm]{Proposition}
\theoremstyle{definition}
\newtheorem{remark}[thm]{Remark}
\newtheorem{Rem-Def}[thm]{Remark-Definition}

\newtheorem{claim}[thm]{Claim}

\newtheorem{example}[thm]{Example}
\newtheorem{s:examples}[thm]{s:examples}

\newtheorem{question}[thm]{Question}
\newtheorem{construction}[thm]{Construction}

\numberwithin{equation}{section}

\newcommand{\ga}[2]{\begin{gather}\label{#1}#2 \end{gather}}

\newcommand{\End}{{\rm End}}

\newcommand{\C}{{\mathbb C}}

\renewcommand{\P}{{\mathbb P}}

\newcommand{\R}{{\mathbb R}}

\newcommand{\Z}{{\mathbb Z}}

\newcommand{\id}{{\rm id\hspace{.1ex}}}

\newcommand{\Aut}{{\rm Aut\hspace{.1ex}}}

\newcommand{\Nef}{{\rm Nef\hspace{.1ex}}}

\newcommand{\GL}{{\rm GL\hspace{.1ex}}}

\newcommand{\NS}{{\rm NS}}

\newcommand{\GK}{{\rm GKdim}\,}
\newcommand{\rGK}{{\rm rGKdim}\,}

\newcommand{\ssec}{\subsection}

\newcommand{\ol}{\overline}
\newcommand{\ti}[1]{\tilde{#1}}

\newcommand{\vast}{\bBigg@{4}}
\newcommand{\Vast}{\bBigg@{5}}
\newcommand{\wt}{\widetilde}

\newcommand{\bk}{\mathbf{k}}

\newcommand{\cC}{\mathcal{C}}

\newcommand{\cK}{\mathcal{K}}

\newcommand{\cO}{\mathcal{O}}

\newcommand{\gD}{\Delta}

\newcommand{\gO}{\Omega}
\newcommand{\gS}{\Sigma}

\renewcommand{\ga}{\alpha}
\newcommand{\gaa}{\alpha}
\newcommand{\gb}{\beta}

\newcommand{\gep}{\varepsilon}

\newcommand{\gk}{\kappa}
\newcommand{\gl}{\lambda}
\newcommand{\go}{\omega}
\newcommand{\gs}{\sigma}

\newcommand{\Id}{\mathrm{Id}}

\newcommand{\Int}{\mathrm{Int}}

\newcommand{\lov}{\mathrm{lov}}

\newcommand{\Plov}{\mathrm{Plov}}

\newcommand{\pr}{\mathrm{pr}}

\newcommand{\topp}{\mathrm{top}}

\newcommand{\torsion}{\mathrm{torsion}}

\newcommand{\Vol}{\mathrm{Vol}}

\newcommand{\cf}{\emph{cf.} }
\newcommand{\eg}{\emph{e.g.} }

\newcommand{\bss}{\backslash}

\newcommand{\cnec}{\mathrel{:=}}
\renewcommand{\(}{\left(}
\renewcommand{\)}{\right)}

\newcommand{\cto}{\circlearrowleft}

\newcommand{\dto}{\dashrightarrow}

\newcommand*\eto{%
	\xrightarrow[]{\raisebox{-0.25 em}{\smash{\ensuremath{\sim}}}}%
}
\newcommand{\hto}{\hookrightarrow}

\makeatletter
\@namedef{subjclassname@2020}{\textup{2020} Mathematics Subject Classification}
\makeatother

\title[Polynomial log-volume growth and GK-dimension]
{Polynomial log-volume growth 
	and
the GK-dimensions of twisted homogeneous coordinate rings
}

\author{Hsueh-Yung Lin}
\address{Department of Mathematics, National Taiwan University, 
	and National Center for Theoretical Sciences,
	Taipei, Taiwan.}
\email{hsuehyunglin@ntu.edu.tw}

\author{Keiji Oguiso}
\address{Mathematical Sciences, the University of Tokyo, Meguro Komaba 3-8-1,
Tokyo, Japan, and National Center for Theoretical Sciences, Mathematics Division, National Taiwan University,
Taipei, Taiwan}
\email{oguiso@ms.u-tokyo.ac.jp}

\author{De-Qi Zhang}
\address{Department of Mathematics, National University
of Singapore, 10 Lower Kent Ridge Road, Singapore 119076,
Republic of Singapore}
\email{matzdq@nus.edu.sg}

\medskip

\date{}

\begin{document}

\begin{abstract}
	Let $f$ be a zero entropy automorphism of a compact K\"ahler manifold $X$.
	We study the polynomial log-volume growth $\Plov(f)$ of $f$
	in light of the dynamical filtrations introduced in our previous work with T.-C. Dinh.
	We obtain
	new upper bounds and lower bounds of $\Plov(f)$.
	As a corollary, we
	completely determine $\Plov(f)$
	when $\dim X = 3$,
	extending a result of Artin--Van den Bergh for surfaces.
	When $X$ is projective, $\Plov(f) + 1$ coincides with
	the Gelfand--Kirillov dimensions $\GK (X,f)$ of 
	the twisted homogeneous coordinate rings associated to $(X,f)$.
	Reformulating these results for $\GK (X,f)$, 
	we improve Keeler's bounds of $\GK (X,f)$
	and provide effective upper bounds of 
	$\GK (X,f)$ which only depend on $\dim X$.
\end{abstract}

%

\maketitle

\tableofcontents

\section{Introduction} \label{s:intro}
\subsection{Zero entropy automorphisms}\label{introPlov}
\hfill

Let $X$ be a compact K\"ahler manifold and 
let $f : X \cto$
be an automorphism (\emph{i.e.} biholomorphic self-map)  of $X$.
The topological entropy $h_\topp(f)$ is an invariant
measuring the complexity of the dynamical system $f : X \cto$.
Thanks to Gromov--Yomdin's theorem~\cite{Gromov, Yomdin},
we have
\begin{equation}\label{eqn-GY}
	h_\topp(f) = \log r(f) \ge 0,
\end{equation}
where $r(f)$ is the spectral radius of $f^* : H^\bullet(X,\C) \cto$.

This paper focuses on automorphisms $f$ with \emph{zero entropy}  $h_\topp(f) = 0$ (cf.~Lemma~\ref{lem:equiv_dyn}).
In the context of complex dynamics of compact K\"ahler manifolds, 
they have recently been investigated in various works 
(see \eg \cite{CantatICM, Lo, CO,  FanFuOuchi, DLOZ}).
In these works,
more refined invariants of them are studied,  
such as the polynomial entropy, the polynomial log-volume growth $\Plov(f)$~\cite{CO}, and 
{\it the polynomial growth} $k(f)$
of the pullbacks~\cite{Lo, DLOZ}: 
$$\|(f^{m})^*: H^{1,1}(X)\circlearrowleft\| \asymp_{m \to \infty} m^{k(f)}.$$
New structures of $f^* : H^\bullet(X,\C) \cto$
have also been discovered such as the
dynamical filtrations~\cite[\S 3]{DLOZ}. 
Below is one such consequence, 
which is also relevant to the present work.

\begin{theorem}[{\cite[Theorem 1.1, Remark 3.9.(1)]{DLOZ} and Corollary~\ref{cor-JBboundkd}}]\label{JBbound}
	Let $f \in {\rm Aut}\, (X)$ be an automorphism of zero entropy.
	Assume that $d \cnec \dim X \ge 1$.
	Then $k(f)$ is an even integer which satisfies
	\begin{equation}\label{ub-DLOZ}
		k(f) \le 2(d -1) 
	\end{equation}
	and
	\begin{equation}\label{ub-DLOZgk}
		k(f) \le 2(d -\gk(X)),
	\end{equation}
	where $\gk(X)$ is the Kodaira dimension of $X$.
	Moreover, these estimates are optimal.
\end{theorem}

The upper bound~\eqref{ub-DLOZ} is the most essential part 
and was proven in~\cite[Theorem 1.1]{DLOZ}.
We will prove 
the refinement~\eqref{ub-DLOZgk} 
in Corollary~\ref{cor-JBboundkd},
based on the approach developed in~\cite{DLOZ}.

\subsection{Polynomial log-volume growths}
\hfill

The main goal of this paper is to study the {\it polynomial log-volume growth} $\Plov(f)$
of an automorphism $f : X \cto$. We first recall its definition.
For every $n \ge 1$, let $\Gamma_f(n) \subset X^{n+1}$ be the graph of 
$$f \times f^2 \times \cdots \times f^n : X \to X^n$$ 
and let $\Vol_\go(\Gamma_f(n))$ be the volume of $\Gamma_f(n)$ 
defined with respect to a K\"ahler form $\go$ on $X$.
We then define
$$\Plov(f) \cnec \Plov(X, f) \cnec \limsup_{n \to \infty} \frac{\log \Vol_\go(\Gamma_f(n))}{\log n}
\in [0,\infty].$$
This invariant of $f$ is 
independent of the choice of $\go$ (Lemma~\ref{lem-sminvVol}). 

We will study upper bounds and lower bounds of $\Plov(f)$ 
in terms of $d = \dim X$ and $k(f)$ introduced in \S\ref{introPlov}.
Using dynamical filtrations,
we obtain the following estimates.

\begin{thm}\label{uniform}
	Let $X$ be a compact K\"ahler manifold of dimension $d$ and 
	let $f \in \Aut (X)$ be a zero entropy automorphism.
	\begin{enumerate}
		\setcounter{enumi}{-1}
		\item (Corollary~\ref{cor-k=0}) $\Plov(f) = d$ if and only if $k(f) = 0$.
		\item (Proposition~\ref{pro-KUB} and Theorem~\ref{thm-lb}) Suppose that $k(f) > 0$. Then we have
		$$d + 2k(f) - 2 \le \Plov (f) \le k(f)(d-1) + d.$$
		\item (Theorem~\ref{thm-uniform}) Suppose that $k(f) > 0$ and $d \ge 3$. Then we have
		$$ \Plov(f) \le  k(f)(d-1) + d - 2. $$
	\end{enumerate}
\end{thm}

By Theorem~\ref{JBbound}, we have
$$k(f) \in \Set{0,2,\ldots,2d-2}.$$
Also $\Plov(f)$ has the same parity as $\dim X$; see Corollary~\ref{cor-parity}.
These together with Theorem~\ref{uniform} immediately determine $\Plov(f)$ when $d = 2,3$.

\begin{cor}\label{cor-d=3}
	\hfill
	\begin{enumerate}
		\item  If $d = 2$, then
		$$\Plov(f) = 
		\begin{cases}
			2 \text{ if } k(f) = 0 \\ 
			4 \text{ if } k(f) = 2. 
		\end{cases}
		$$
		\item If $d = 3$, then
		$$\Plov(f) = 
		\begin{cases}
			3 \text{ if } k(f) = 0 \\
			5 \text{ if } k(f) = 2 \\  
			9 \text{ if } k(f) = 4. 
		\end{cases}
		$$
	\end{enumerate}
\end{cor}

Together with Theorem~\ref{JBbound}, 
Theorem~\ref{uniform} implies that 
$\Plov(f) \le 2d^2 - 3d$
whenever $d \ge 3$.
When $d \ge 4$, we will further improve this upper bound  to 
\begin{equation}\label{ineq-dge4}
	\Plov(f) \le 2d^2-3d - 2;
\end{equation}
see Proposition~\ref{uniformdge4}.

\subsection{A conjectural upper bound}
\hfill

When $X$ is a complex torus, we determine $\Plov(f)$
in terms of the pullback $f^* : H^{1,0}(X)  \cto$.

\begin{theorem}\label{thm51}
	Let $X$ be a complex torus of dimension $d$ and $f \in \Aut (X)$ an automorphism
	of zero entropy.
	Assume that the Jordan canonical form of $f^* : H^{1,0}(X)\cto$
	consists of Jordan blocks of sizes $k_1,\ldots,k_p$, counted with multiplicities. Then
	$$\Plov(f) = \sum_{i=1}^{p} k_i^2.$$
	In particular, we have $\Plov(f) \le d^2$, and this upper bound is optimal among complex tori.
\end{theorem}

Theorem~\ref{thm51} also shows that 
the quadratic {\it order} of the upper bounds with respect to $d$ 
in~\eqref{ineq-dge4} is optimal. We will also compute 
$\Plov (X, f)$ for other examples including threefolds; see \S\ref{sec-MainTheorem2}.
As we fail to construct examples of $f: X \cto$ such that $d^2 < \Plov(f) < \infty$ where $d = \dim X$,
presumably the upper bound in~\eqref{ineq-dge4} when $d \ge 4$ is still not optimal.
Taking Corollary~\ref{cor-d=3} and Theorem~\ref{thm51} 
into account, 
it seems reasonable to ask the following questions.

\begin{question}\label{mainquest2}
	Let $X$ be a compact K\"ahler manifold of dimension $d \ge 1$. Let $f \in \Aut (X)$ be a zero entropy automorphism.
	\begin{enumerate}
		\item Is $\Plov (f) \le d^2$?
		\item More precisely, are possible values of $\Plov(X,f)$ always
		realizable by $d$-dimensional tori? 
	\end{enumerate}
\end{question}

Question~\ref{mainquest2} (1) is the analogous question of~\cite[Question 4.1]{CO}
asked for polynomial entropy by Cantat--Paris-Romaskevich.
By~\cite[(2.7)]{CO}, a positive answer to Question~\ref{mainquest2} (1)
also answers~\cite[Question 4.1]{CO} in the affirmative.

The following partial answer to Question~\ref{mainquest2}
is a direct consequence of the above theorems.

\begin{cor}\label{cor-kod}  (see \S \ref{sec-lb})
	Let $X$ be a compact K\"ahler manifold of dimension $d \ge 3$
	and let $f \in \Aut (X)$ be a zero entropy automorphism.
	\begin{enumerate}
		\item If $k(f) \le d$, then $\Plov (f) \le d^2 - 2$.
		In particular, $\Plov (f) \le  d^2 - 2$, whenever $\kappa (X) \ge d/2$.
		\item Question~\ref{mainquest2} 
		has positive answers 
		when $\dim X \le 3$.
	\end{enumerate}
\end{cor}

Recently, Hu and Jiang~\cite{HuJiang} answered Question~\ref{mainquest2} (1) in the affirmative.

\subsection{Gelfand--Kirillov dimension}\label{ssec-GK}
\hfill

When $X$ is projective, 
the polynomial log-volume growth $\Plov(f)$
actually coincides with
some known invariant of $f$
studied in noncommutative algebra.
The following identification is implicit in the seminal work of Keeler~\cite{Ke}. 

\begin{thm}\label{thm-comp}
	Let $X$ be a smooth projective variety defined over an algebraically closed field,
	and let $f \in \Aut(X)$ be a zero entropy automorphism.
	Then
	$$\GK B(X,f,L) - 1 = \Plov(X,f).$$
	Here, $\GK B(X,f,L)$ is the 
	Gelfand--Kirillov dimension (or GK-dimension for short)
	of the twisted homogeneous coordinate ring $B(X,f,L)$
	associated to $f : X \cto $ and any ample line bundle $L$. 
\end{thm}
We refer to Section~\ref{sec-THCR} for the definition and basic properties
of $\GK B(X,f,L)$,
as well as the proof of Theorem~\ref{thm-comp}.
In this regard, 
two of our results are not new for projective varieties.
The first one is the
upper bound in Theorem~\ref{uniform}.(1),
as the estimate  
$$\GK B(X,f,L) -1 \le k(f)(d-1) + d$$
has already been proven in~\cite[Lemma 6.13]{Ke}.
The second one is Corollary~\ref{cor-d=3} (1), due to
Artin--Van den Bergh~\cite[Theorem 1.7]{AvB}.
Our approach based on dynamical filtrations
is however completely different, and extends both results in a non-trivial way.

Thanks to Theorem~\ref{thm-comp},
the main results we prove for $\Plov(f)$ also translate to
new results about the GK-dimension of $B(X,f,L)$;
see Corollary~\ref{uniformGK} for some instances.
So far, the GK-dimension has been studied 
mostly by specialists in noncommutative algebras.
We hope that the dynamical properties of $(X,f)$ might provide a better understanding of
the algebraic structure of $B(X,f,L)$, and \emph{vice versa}.

\subsection{Organization of the paper and a few remarks to the readers}
\hfill

We start with
Section~\ref{sec-plov}, proving basic properties of polynomial log-volume growth
(see e.g. Proposition~\ref{prop1-Kahler}). 
In Section~\ref{sec-filtr}, 
we recall the construction of quasi-nef sequences and dynamical filtrations,
together with their fundamental properties.
We also prove several auxiliary results related to them, 
which will be useful in the study of 
upper and lower bounds of $\Plov(f)$.
Section~\ref{sec-filtr} also contains a proof of the statement in Theorem~\ref{JBbound} 
involving the Kodaira dimension.
Section~\ref{proofuniform} and Section~\ref{sec-lb} are 
devoted to upper and lower bounds of $\Plov(f)$ respectively, 
all together implying Theorem~\ref{uniform} and Corollary \ref{cor-kod}.
In
Section~\ref{sec-MainTheorem3} and Section~\ref{sec-MainTheorem2}, we study explicit examples,
which contain
complete descriptions of $\Plov(f)$ for tori (Theorem~\ref{thm51}).
Section~\ref{sec-THCR} starts with a brief review of twisted homogeneous coordinate rings
and their GK-dimensions.
We recall some fundamental results proven in~\cite{Ke,AvB} (Theorem~\ref{Keeler})
and derive Theorem \ref{thm-comp} as a direct consequence.
We finish Section~\ref{sec-THCR} by Corollary~\ref{uniformGK}, translating results from $\Plov(f)$ to GK-dimensions.

\subsection*{ Notations and conventions}
\hfill

All manifolds are assumed to be connected.
Let $X$ be a compact K\"ahler manifold.
Write
$$H^{i,i}(X,\R) \cnec H^{i,i}(X) \cap H^{2i}(X,\R).$$
For every $\gaa \in H^{i,i}(X,\R)$, if $\gaa \cdot H^{1,1}(X,\R)^{d-i} = 0$ (where $d = \dim X$), 
we write
$$\gaa \equiv 0.$$
We follow \cite{DemaillyMethod} for the basic terminology, like positive classes and cones.

\section{Polynomial log-volume growth} \label{sec-plov}

\ssec{Definition and basic properties}\hfill

Let $X$ be a compact K\"ahler manifold of dimension $d \ge 1$ and
let $f \in \Aut(X)$. Let
$\go \in H^{1,1}(X,\R)$ be a K\"ahler class.
For every integer $n \ge 1$,
the volume of the graph  $\Gamma_f(n) \subset X^{n+1}$
of 
$$f \times f^2 \times \cdots \times f^n : X \to X^n$$
with respect to any K\"ahler metric in the class $\go$
is equal to
$$
\Vol_\go(\Gamma_f(n)) = \int_{\Gamma_f(n)} \frac{1}{d!}\(\sum_{i = 1}^{n+1} \pr_i^*\go\)^d
= \frac{1}{d!}\Delta_n(f, \go)^d,
$$
where $\pr_i : X^{n+1} \to X$ is the projection to the $i$th factor and
$$\Delta_n(f, \go) \cnec \sum_{i = 0}^n (f^i)^*\go \in H^{1,1}(X,\R).$$

Note that the class $\Delta_n(f, \go)$ and the invariant $\Vol_\go(\Gamma_f(n))$ are defined
more generally for any $\go \in H^{1,1}(X,\R)$.
But in order to define $\Plov(f,\go)$ below,
the class $\go$ 
needs to satisfy
$\Delta_n(f, \go)^d \ge 0$.
A natural sufficient condition is that $\go$ is nef.

\begin{lem}\label{lem-sminvVol}
	For every nef $\ga \in \Nef(X) \subset H^{1,1}(X,\R)$, define
	$$\Plov(f,\ga) \cnec \limsup_{n \to \infty} \frac{\log \Vol_\ga(\Gamma_f(n))}{\log n}
	=  \limsup_{n \to \infty} \frac{\log \Delta_n(f, \ga)^d}{\log n}  \in  \R \cup \{-\infty,\infty\},$$
	where we set $\log 0 \cnec -\infty$.
	Then $\Plov(f,\go)$ is independent of $\go$ whenever $\go$ is nef and big, and we have $\Plov(f,\go) \ge 1$.
\end{lem}

Lemma~\ref{lem-sminvVol} justifies the well-definedness of
the polynomial logarithmic volume growth of $f$ in the introduction, which
is defined to be
$$\Plov(f) \cnec \Plov (X, f) \cnec \Plov(f,\go),$$
where $\go$ is any nef and big class.
We refer to Corollary~\ref{cor-nefbdry} for an improvement.

To prove Lemma \ref{lem-sminvVol}, we need the following easy but useful result.

\begin{lemma}\label{lem:nef_compar}
	Let $X$ be a compact K\"ahler manifold 
	of dimension $d$,
	and let
	$$L_1\,\, ,\,\, \ldots\,\, ,\,\,  L_d\,\, ,\,\, M_1\,\, ,\,\, \ldots\,\, ,\,\, M_d$$
	be nef classes in $H^{1,1}(X,\R)$ 
	such that $M_i \ge L_i$, i.e., $M_i - L_i$ is pseudo-effective.
	Then
	$$(M_1 \cdots M_d) \ge (L_1 \cdots L_d) .$$
	In particular, $(M_1^d) \ge (L_1^d)$.
\end{lemma}

\begin{proof}
	Inductively, we have
	$$(M_1 \cdots M_d) \ge (L_1 \cdot M_2 \cdots M_d) \ge \cdots
	\ge
	(L_1 \cdots L_j \cdot M_{j+1} \cdots M_d) \ge \cdots
	\ge (L_1 \cdots L_d),$$
	which proves Lemma~\ref{lem:nef_compar}.
\end{proof}

\begin{proof}[Proof of Lemma~\ref{lem-sminvVol}]
	Let $\go$ and $\go'$ be two nef and big classes.
	Then there exists some $\gep > 0$ such that $\go - \gep \go'$ is pseudo-effective.
	Accordingly, $\gD_n(f,\go) - \gep\gD_n(f,\go')$ is pseudo-effective,
	so $\gD_n(f,\go)^d \ge \gep^d\gD_n(f,\go')^d$ by Lemma~\ref{lem:nef_compar},
	and therefore $\Plov(f,\go) \ge \Plov(f,\go')$.
	By symmetry, we have $\Plov(f,\go) = \Plov(f,\go')$.
	
	Finally, since $\go$ is big and nef,
	we have
	$$\Delta_n(f, \go)^d = (\sum_{i = 0}^n (f^i)^*\go)^d 
	\ge \sum_{i = 1}^n ((f^i)^*\go)^d = n\go^d > 0.$$
	Hence
	$$\Plov(f,\go) =  \limsup_{n \to \infty} \frac{\log \Delta_n(f, \go)^d}{\log n}
	\ge  \limsup_{n \to \infty} \frac{\log (\go^d) + \log n }{\log n} = 1.$$
\end{proof}

The following is an immediate consequence of Lemma~\ref{lem-sminvVol}.

\begin{cor}\label{cor-plovinv}
	Let $X$ and $Y$ be compact K\"ahler manifolds
	with automorphisms $f \in \Aut(X)$ and $g \in \Aut(Y)$.
	Suppose that there exists a $\C$-linear isomorphism 
	$$\phi : H^\bullet(X,\C) \eto H^\bullet(Y,\C)$$
	of the cohomology rings such that the following conditions are satisfied:
	\begin{itemize}
		\item[(i)] $\phi \circ f^* = g^* \circ \phi$;
		\item[(ii)] There exists a K\"ahler class $\go \in H^{1,1}(X)$ on $X$
		such that $\phi(\go)$ is K\"ahler  on $Y$.
	\end{itemize}
	Then
	$$\Plov(f) = \Plov(g).$$
	The similar statement holds if $\phi$ is replaced by a $\C$-linear isomorphism of 
	the subalgebras
	$$\psi : \bigoplus_i H^{i,i}(X) \eto \bigoplus_i H^{i,i}(X').$$
\end{cor}

The same argument as in the proof of Lemma~\ref{lem:nef_compar} proves the following.

\begin{lem}\label{lem-nefle} 
	For every nef class $\alpha \in H^{1,1}(X, \R)$,
	we have
	$$\Plov(f,\alpha) \le \Plov(f).$$
\end{lem}

\begin{proof}
	Take a K\"ahler class $\omega$ such that $\go \ge \alpha$.
	By Lemma~\ref{lem:nef_compar}, we have
	$\gD_n(f,\go)^d \ge
	\gD_n(f,\gaa)^d$
	for every integer $n \ge 0$.
	Hence $\Plov(f,\alpha) \le \Plov(f)$.
\end{proof}

Now we prove some basic dynamical properties of $\Plov(X, f)$
summarised in the following,
which will be frequently used in this paper.

\begin{pro}\label{prop1-Kahler}	
	Let $f : X \cto$ be an automorphism of a compact K\"ahler manifold. 	
	\begin{enumerate}
		\item (Independence of the metric and positivity; Lemma~\ref{lem-sminvVol})
		The invariant $\Plov(f)$ is independent of $\go \in H^{1,1}(X,\R)$
		whenever $\go$ is nef and big, and we have $\Plov(f) \ge 1$.
		\item (Finiteness and integrality; Lemmas~\ref{lem:equiv_dyn} and \ref{lem35})
		We have $\Plov(f) < \infty$ if and only if
		$f^* : H^{1,1}(X) \cto$ is quasi-unipotent.
		In this case, $\Plov(f)$ is an integer.
		\item (Finite index; Lemma~\ref{lem-invfini})
		We have $\Plov(f) = \Plov(f^N)$
		for any integer $N \not= 0$.
		\item (Product; Lemma~\ref{lem-pruduct}) Let $X_i$ ($i=1$, $2$) be compact K\"ahler manifolds
		and let $f_i \in \Aut (X_i)$.
		Then
		$$\Plov (f_1 \times f_2) = \Plov (f_1) + \Plov (f_2) $$
		for the product automorphism $f_1 \times f_2 \in \Aut(X_1 \times X_2)$.
		\item (Invariance under generically finite maps; Lemma~\ref{lem-invft}) Let $X$ and $Y$ be
		compact K\"ahler manifolds and
		$f_X \in \Aut (X)$ and $f_Y \in \Aut (Y)$.
		Let $\phi : X \dasharrow Y$
		be a generically finite dominant meromorphic map such that $f_Y \circ \phi = \phi \circ f_X$. Then 	
		$$\Plov(f_X) = \Plov(f_Y).$$	
		\item (Restriction; see Lemma~\ref{lem-restr}, also for the precise 
		definition of $\Plov(f_{|W})$ when $W$ is not smooth)
		Let $W \subset X$ be a closed subvariety such that $f(W) = W$.
		Then $\Plov(f|_{W}) \le \Plov(f)$ for the automorphism $f|_{W} \in \Aut (W)$ induced from $f$ by restriction.
	\end{enumerate}
\end{pro}

First, we prove that $\Plov(X, f)$ is invariant under taking finite iterations.

\begin{lem}\label{lem-invfini}
	For every integer $N \ne 0$, we have $\Plov(f^N) = \Plov(f)$.
\end{lem}

\begin{proof}
	Since
	$$\(\sum_{i = 0}^n (f^{-i})^*\go\)^d
	= \((f^{-n})^*\sum_{i = 0}^n (f^i)^*\go\)^d
	= \(\sum_{i = 0}^n (f^i)^*\go\)^d,$$
	we have $\Plov(f^{-1}) = \Plov(f)$.
	So it suffices to prove Lemma~\ref{lem-invfini} for $N > 0$.
	
	For every integers $r$ and $j>0$ such that $0 \le r < N$,
	consider the K\"ahler form
	$\go_{r,j} \cnec \sum_{i = r}^{r + j -1} (f^{i})^*\go$.
	Then
	$$\go_{r,(m+1)N} \le \go_{0,r} + \go_{r,(m+1)N} = \go_{0,r + (m+1)N} \le \go_{r-N,(m+2)N}.$$
	So
	$$\Vol_{\go_{r,(m+1)N}}(X) \le \Vol_{\go_{0,r + (m+1)N}}(X) \le \Vol_{\go_{r-N, (m+2)N}}(X)$$
	by Lemma~\ref{lem:nef_compar}, and thus
	$$\Vol_{\sum_{j= 0}^{N-1}(f^{r+j})^*\go}(\Gamma_{f^N}(m))
	\le \Vol_{\go}(\Gamma_{f}(r + mN))
	\le \Vol_{\sum_{j= 1}^{N}(f^{r-j})^*\go}(\Gamma_{f^N}(m + 1)).$$
	By Lemma~\ref{lem-sminvVol}, we have
	$$\limsup_{m \to \infty}\frac{\log \Vol_{\sum_{j= 0}^{N-1}(f^{r+j})^*\go}(\Gamma_{f^N}(m))}{\log m}
	= \Plov(f^N) =
	\limsup_{m \to \infty}\frac{\log \Vol_{\sum_{j= 1}^{N}(f^{r-j})^*\go}(\Gamma_{f^N}(m + 1))}{\log (m+ 1)},$$
	so for every integer $r$ such that $0 \le r < N$, we have
	$$
	\limsup_{m \to \infty}\frac{\log \Vol_{\go}(\Gamma_{f}(r + mN))}{\log (r+mN)}
	= \limsup_{m \to \infty}\frac{\log \Vol_{\go}(\Gamma_{f}(r + mN))}{\log (m)}
	= \Plov(f).
	$$
	Hence $\Plov(f) = \Plov(f^N)$.
\end{proof}

\begin{cor}\label{cor-nefbdry}
	Let $\ga \in \Nef(X)$. We have
	$$\Plov(f,\ga) = \Plov(f) \in [1,\infty]$$ 
	as long as $\Plov(f,\ga) \ne -\infty$.
\end{cor}

\begin{proof}
	
	Suppose that $\Plov(f,\ga) \ne -\infty$.
	Then $ \Delta_N(f, \ga)^d > 0$ for some  integer $N \ge 0$.
	Since $\go \cnec \Delta_N(f, \ga)$ is nef, it is thus big. Using Lemma~\ref{lem-invfini},
	we have
	\begin{equation*}
		\begin{split}
			\Plov(f,\ga) & = \limsup_{n \to \infty} \frac{\log \Delta_n(f, \ga)^d}{\log n} \\
			& \ge \limsup_{k \to \infty} \frac{\log \Delta_{Nk-1}(f, \ga)^d}{\log (Nk - 1)}
			= \limsup_{k \to \infty} \frac{\log \Delta_{k}(f^N, \go)^d}{\log k } = \Plov(f^N) = \Plov(f) 
		\end{split}
	\end{equation*}
	It follows from Lemma~\ref{lem-nefle} that $\Plov(f,\ga) = \Plov(f)$.
\end{proof}

We can characterize whether a holomorphic automorphism $f \in \Aut(X)$ has
zero entropy based on the finiteness of $\Plov(f)$.

\begin{lemma}\label{lem:equiv_dyn}
	Let $X$ be a compact K\"ahler manifold of dimension $d \ge 1$ and let $f \in \Aut(X)$. 
	Then the following conditions are equivalent.
	
	\begin{enumerate}
		\item
		$f^* : {H^{\bullet}(X, \C)} \cto$ is quasi-unipotent, i.e., a positive power of it is unipotent.
		\item
		$f^* : {H^{1,1}(X)} \cto$ is quasi-unipotent.
		\item
		the first dynamical degree $d_1(f) = 1$.
		\item
		the topological entropy $h_{\topp}(f) = 0$.
		\item The growth of $ \Vol_\go(\Gamma_f(n))$ for any K\"ahler class $\go$
		is sub-exponential, namely
		$$\limsup_{n \to \infty} \Vol_\go(\Gamma_f(n))^{1/n} = 1.$$
		\item
		$\Plov(f) < \infty$. In other words,
		the growth of $ \Vol_\go(\Gamma_f(n))$ for any K\"ahler class $\go$
		is polynomial.
	\end{enumerate}
\end{lemma}

Here we recall that for $1 \le i \le d$,
the $i$-th {\it dynamical degree} of $f$ is defined as
\begin{equation}\label{eqn-projddeg}
	d_i(f) := \lim_{n \to \infty} (\omega^{d - i} \cdot (f^n)^*\omega^i)^{1/n},
\end{equation}
where $\omega \in H^{1,1}(X)$ is a K\"ahler class~\cite{DS3}; 
these $d_i(f)$ are independent of $\go$. 	

\begin{proof}
	The equivalence of the first five conditions is well-known and is obtained as follows.
	By Gromov--Yomdin's theorem (\cf \cite{Gromov}, \cite{Yomdin}; see also~\cite[Theorem 3.6]{OguisoICM}), we have
	$$h_{\topp}(f) = \lov(f) = \log r(f^*) = \log(\max_{1 \le i \le d} \, \{d_i(f)\}),$$
	where $r(f^*)$  is the spectral radius of $f^* : H^\bullet(X,\C) \circlearrowleft$, and
	$$\lov(f) \cnec  \limsup_{n \to \infty} \frac{\log \Vol_\go(\Gamma_f(n))}{n}.$$
	Together with the log concavity of dynamical degrees $d_i(f)$ 
	(which follows from Khovanskii--Teissier's inequality), this
	implies that $h_{\topp}(f) > 0$ if and only if $d_i(f) > 1$ for some (and hence all)
	$i \in \{1, \cdots, d-1\}$.
	Thus the equivalence of the first five assertions follows from Kronecker's theorem. 
	Also, since  
	$$\frac{\log n}{n} \cdot \Plov(f) \ge \lov(f)  = h_{\topp}(f) \ge 0$$
	for all $n>1$,
	(6) implies
	these assertions.
	
	To see that (2) implies (6),
	recall that in order to compute $\Plov(f)$, by Lemma~\ref{lem-invfini}
	we can replace $f$ by some iteration of it,
	so that $f^* : H^{1,1}(X) \circlearrowleft$ is unipotent.
	Hence $\Plov(f) < \infty$ is a consequence of Lemma~\ref{lem35} below.
\end{proof}

Next, we prove the invariance of $\Plov(f)$ under generically finite meromorphic maps.

\begin{lemma}\label{lem-invft}
	Let $X$ and $Y$ be
	compact K\"ahler manifolds of dimension $d \ge 1$ and
	$f_X \in \Aut (X)$ and $f_Y \in \Aut (Y)$.
	Let $\phi : X \dasharrow Y$
	be a generically finite dominant meromorphic map such that $f_Y \circ \phi = \phi \circ f_X$.
	Then 
	$$\Plov(f_X) = \Plov(f_Y).$$
\end{lemma}

\begin{proof}
	First we reduce to the case where $\phi$ is holomorphic.
	Let $\Gamma$ be the Zariski closure of the graph of $\phi$ in $X \times Y$. Let $p_X : \Gamma \to X$ and $p_Y : \Gamma \to Y$ be the projection. Since $f_X \in \Aut(X)$ and $f_Y \in \Aut(Y)$, it follows that
	$$f_{\Gamma} := (f_X \times f_Y)|_{\Gamma} \in \Aut (\Gamma)$$
	and $f_{\Gamma}$ and $f_X$ (resp. $f_{\Gamma}$ and $f_Y$) are equivariant with respect to a generically finite surjective morphism $p_X$ (resp. $p_Y$).
	By the existence of functorial resolution of singularities~\cite{BM} (see also~\cite[Theorem 3.45]{KollarResSing}),
	there exists a K\"ahler desingularization
	$\nu : \wt{\Gamma} \to \Gamma$
	such that
	$f_\Gamma \circ \nu = \nu \circ f_{\wt{\Gamma}}$
	for some $f_{\wt{\Gamma}} \in \Aut(\wt{\Gamma})$.
	If Lemma~\ref{lem-invft} holds whenever $\phi$ is holomorphic,
	then  $ \Plov(f_X) = \Plov(f_{\wt{\Gamma}}) = \Plov(f_Y)$.
	
	For every K\"ahler class $\go$ on $Y$, since
	$$
	\(\sum_{i = 0}^n (f_X^i)^*(\phi^*\go)\)^d
	= \phi^*\(\sum_{i = 0}^n (f_Y^i)^*\go\)^d
	= \deg(\phi) \cdot \(\sum_{i = 0}^n (f_Y^i)^*\go\)^d,$$
	we have $\Plov(f_X,\phi^*\go) = \Plov(f_Y,\go) = \Plov(f_Y)$.
	As $\phi^*\go$ is nef and big, it follows from Lemma~\ref{lem-sminvVol}
	that $\Plov(f_X) = \Plov(f_Y)$.
\end{proof}

\begin{lem}\label{lem-restr}
	Let $W \subset X$ be a closed subvariety such that $f(W) = W$.
	Then
	$\Plov(f|_{W}) \le \Plov(f)$ for the automorphism $f|_{W} \in \Aut (W)$ induced from $f$ by restriction. Here we define
	$$\Plov(f|_{W}) := \Plov(\wt{W}, \tilde{f}),$$
	where $\tau : \wt{W} \to W$
	is any K\"ahler desingularization of $W$ such that
	$f|_{W} \circ \tau = \tau \circ \ti{f}$ for some $\tilde{f} \in \Aut (\wt{W})$,
	which does not depend on the choice of $\widetilde{W}$ by Lemma \ref{lem-invft}.
\end{lem}

\begin{proof}
	Let $\nu: \wt{W} \to X$ be the composition of $\tau$ with the inclusion $W \hto X$.
	Let $d \cnec \dim X$ and $e \cnec \dim W$.
	Let $\go$ be a K\"ahler class of $X$.
	Up to replacing $\go$ by some positive multiple of it,
	we can assume that
	$$(\go^{d-e} - [W]) \cdot \gb \ge 0$$
	for every $\gb$ in the closed convex cone generated by products of
	$e$ K\"ahler classes.
	
	Since $\sum_{i = 0}^n(f^i)^*(\go) - \go$ is K\"ahler when $n \ge 1$,
	by Lemma~\ref{lem:nef_compar} we have,
	\begin{equation}
		\begin{split}
			\(\sum_{i = 0}^n(\ti{f}^i)^*(\nu^*\go)\)^e
			& = \(\(\sum_{i = 0}^n(f^i)^*(\go)\)^e \cdot [W]\) \\
			& \le
			\(\(\sum_{i = 0}^n(f^i)^*(\go)\)^e \cdot \go^{d-e}\) \le \(\sum_{i = 0}^n(f^i)^*(\go)\)^d.
		\end{split}
	\end{equation}
	So $\Plov(\ti{f},\nu^*\go) \le \Plov(f,\go) = \Plov(f)$.
	As $\nu^*\go$ is nef and big, we have $\Plov(\ti{f}) = \Plov(\ti{f},\nu^*\go)$
	by Lemma~\ref{lem-sminvVol}. Hence $\Plov(f|_{W}) \le \Plov(f)$.
\end{proof}

\ssec{Cohomological polynomial growth $k(f)$}\hfill

Assume that $f^* : H^{1,1}(X) \cto$ is unipotent.
The operator 
$$ N\cnec f^* - \Id : H^{1,1}(X) \to H^{1,1}(X)$$
is thus nilpotent, and we define
$$k(f) \cnec  \max \Set{k \in \Z | ( f^* - \Id)^k \ne 0}.$$
Equivalently, $k(f)+1$ is the maximal size of the Jordan blocks
of the Jordan canonical form of $f^* : H^{1,1}(X) \cto$.
If $f^* : H^{1,1}(X) \cto$ is quasi-unipotent, we define
$$k(f) \cnec k(f^M),$$
where $M$ is a positive integer such that $(f^*)^M$ is unipotent; this invariant is independent of $M$.
Finally if $f^* : H^{1,1}(X) \cto$ is not quasi-unipotent, we set $k(f) = \infty$.

The following result implies in particular that $k(f)$ is invariant under bimeromorphic modifications.

\begin{pro}\label{pro-birJB} 
	Let  $\pi : X \dto Y$ be a dominant, generically finite meromorphic map between
	compact K\"ahler manifolds.
	Let $f_X \in \Aut(X)$ and $f_Y \in \Aut(Y)$ be automorphisms such that
	$$\pi \circ f_X = f_Y \circ \pi.$$
	Then
	$$k(f_X) = k(f_Y).$$
\end{pro}

We shall also prove the following.

\begin{lem}\label{lem-kprod}
	Let $X$ and $Y$ be compact K\"ahler manifolds. Let $f_X \in \Aut(X)$ and $f_Y \in \Aut(Y)$.
	Then we have 
	$$k(f_X \times f_Y) = \max \Set{k(f_X), k(f_Y) }.$$
\end{lem}

To prove both  Proposition~\ref{pro-birJB} and Lemma~\ref{lem-kprod}, we need the following result in linear algebra.

\begin{lem}\label{lem-BPFk}	
	Let $V$ be a finite dimensional vector space over $\R$
	and let $\phi \in \GL(V)$ be a unipotent operator. 
	Let $N \cnec \phi - \Id_V$ and let $k$ denote the largest integer such that $N^{k} \ne 0$.
	Assume that $\phi$ preserves a closed salient convex cone
	$\cC \subset V$ with nonempty interior. 
	Then for every $v \in \Int(\cC)$, the following assertions hold.
	
	\begin{enumerate}
		\item  We have $N^{k}(v) \in \cC 
		\bss \{0\}$ and
		\begin{equation*}\label{eqn-||}
			\phi^n(v)  \sim_{n \to \infty} C_vn^{k} \cdot N^{k}(v).
		\end{equation*}
		for some $C_v > 0$.
		\item For every linear form $\chi : V \to \R$ such that $\chi(\cC \bss \{0\}) > 0$, we have
		\begin{equation*}\label{eqn-chi}
			\chi(\phi^n(v)) \sim_{n \to \infty} C'_vn^{k}
		\end{equation*}
		for some $C_v' > 0$.
	\end{enumerate}
	
\end{lem}

\begin{proof} 
	We can assume $\phi \ne \Id$. 
	Then $\ker N^k \ne V$, and
	for every $w \in V \bss \ker N^k$,
	developing $\phi^n(w) = (\Id_V + N)^n(w)$ shows that
	$$\phi^n(w)/n^k \sim_{n \to \infty} C_w N^k(w) $$ 
	for some $C_w > 0$.
	If moreover $w \in \cC$, then $\phi(\cC) \subset \cC$ and $\cC$ being closed, imply
	$N^k(w) \in \cC$.
	Assume the contrary that there exists some $x \in \Int(\cC)$ such that $N^k(x) = 0$.
	Then there exists some $\gep \in V$ such that
	$$x \pm \gep \in \cC \text{ and }  N^k(x \pm \gep) \ne 0.$$
	As $x \pm \gep \in \cC$ and $\phi(\cC) \subset \cC$, both $N^k(x \pm \gep) = \pm N^k(\gep)$
	are limits of elements in $\cC$, which contradicts the assumptions that $\cC$ is closed and salient. 
	This proves (1).
	
	Since $N^k(v) \in \cC \bss \{0\}$, we have $\chi(N^k(v)) > 0$.
	Thus (2) follows from (1).
\end{proof}

\begin{proof}[Proof of Lemma~\ref{lem-kprod}]
	Assume that $k(f_X) = \infty$ or $k(f_Y) = \infty$.
	Then Lemma~\ref{lem-kprod} follows from the product formula of the first dynamical degree (\cite[Theorem 1.1]{DNT}
	together with Lemma~\ref{lem:equiv_dyn}.
	
	Assume that both $k(f_X)$ and $k(f_Y)$ are finite.
	By Lemma~\ref{lem:equiv_dyn}, up to replacing $f_X$ and $f_Y$ by a common positive power,
	we can assume
	$$f_X^* : H^{1,1}(X,\R) \cto, \ \ f_Y^* : H^{1,1}(Y,\R)\cto, \ \ \text{ and } \ \ (f_X \times f_Y)^* : H^{1,1}(X  \times Y,\R)  \cto$$ 
	are unipotent. Fix K\"ahler classes $\go_X \in H^{1,1}(X,\R)$ and $\go_Y \in H^{1,1}(Y,\R)$. 
	Let $p_X : X \times Y \to X$ and $p_Y : X \times Y \to Y$ be the projections.
	Applying Lemma~\ref{lem-BPFk} to $H^{1,1}(\bullet,\R)$ and the nef cone therein shows that 
	$k(f_X \times f_Y)$ (resp. $k(f_X)$ and $k(f_Y)$)
	is the polynomial growth rate of
	$$((f_X \times f_Y)^*)^n(p_X^* \go_X + p_Y^* \go_Y) = p_X^*(f_X^*)^n(\go_X) +p_Y^* (f_Y^*)^n(\go_Y)$$
	(resp. $(f_X^*)^n(\go_X)$ and $(f_Y^*)^n(\go_Y)$). 
	
	Hence $k(f_X \times f_Y) = \max \Set{k(f_X), k(f_Y) }$.
\end{proof}

\begin{proof}[Proof of Proposition~\ref{pro-birJB}]
	
	As in Lemma \ref{lem-invft}, up to replacing $X$
	by an equivariant desingularization of
	the graph of $\pi$, we can assume that $\pi$ is holomorphic.
	
	By Lemmas~\ref{lem-invft} and \ref{lem:equiv_dyn}, we have $k(f_X) = \infty$ if and only if $k(f_Y) = \infty$.
	Thus we can assume that both $f_X^* : {H^{1,1}(X)} \cto$ and $f_Y^* : {H^{1,1}(Y)} \cto$ are quasi-unipotent.
	Up to replacing $f_X$ and $f_Y$ by some positive iterations,
	we can assume that the above actions are both unipotent.
	
	Applying Lemma~\ref{lem-BPFk} to the nef cone
	in $H^{1,1}(X,\R)$, we see that
	for every pair of K\"ahler classes $\omega$, $\eta$ on $X$, we have
	\begin{equation}\label{rateX}
		(f_X^*)^n (\omega) \cdot \eta^{\dim\, X -1}   \sim_{n \to \infty} C n^{k(f_X)}
	\end{equation}
	for some $C >0$.
	Similarly,  for every pair of K\"ahler classes $\go'$, $\eta'$ on $Y$, we have
	\begin{equation}\label{rateY}
		(f_X^*)^n (\pi^*\go')\cdot \pi^*\eta'^{\dim X -1}
		=\deg(\pi) \cdot (f_Y^*)^n (\go') \cdot \eta'^{\dim Y -1}  \sim_{n \to \infty} C' n^{k(f_Y)}
	\end{equation}
	for some $C' > 0$.
	Since the classes $\go,\eta, \pi^*\go',\pi^*\eta'$ are all nef and big,
	with the notation of Lemma~\ref{lem:nef_compar} we have
	$$c_1 \pi^*\go' \le \omega \le c_2 \pi^*\go' \ \text{ and }  \ c_3 \pi^*\eta' \le \eta \le c_4 \pi^*\eta'$$
	for some positive real numbers $c_i$.
	It follows from  Lemma~\ref{lem:nef_compar}
	that the growth rates of~\eqref{rateX} and~\eqref{rateY} have the same polynomial order.
	Hence $k(f_X) = k(f_Y)$.
\end{proof}

\ssec{Bounding the polynomial log-volume growth}\hfill

{\it From now on until the end of Section~\ref{sec-plov}, we assume that
$$f^* : H^{1,1}(X) \cto \text{ is unipotent, }$$ 
unless otherwise specified.}

For every $\gaa \in H^{1,1}(X,\R)$,
recall that 
$$\Delta_n(f, \gaa)   \cnec \sum_{i = 0}^n (f^i)^*\go \in H^{1,1}(X,\R).$$
The following lemma shows that $\Delta_n(f, \gaa)$ has polynomial expressions in $n$ for both ranges
$n \in \Z_{\ge 0}$ and $n \in \Z_{\le 0}$ (but these two polynomials are usually different).

\begin{lem}\label{lem-gD+-} 
	We have
	\begin{equation}
		\Delta_n(f, \gaa) = 
		\begin{cases}
			\gD^+_n(f,\gaa)  & \text{ if } n \ge 0 \\
			-\gD^+_{n-1}(f,\gaa) & \text{ if } n \le 0,
		\end{cases}
	\end{equation}
	where
	$$\gD^+_n(f,\gaa) \cnec \sum_{j=0}^{k(f)}{n + 1\choose j+1} N^j\gaa.$$
\end{lem}

\begin{proof}
	By definition of $k \cnec k(f)$, for every $\gaa \in H^{1,1}(X,\R)$ and every $n \in \Z$:
	\begin{equation}\label{eqn-abel}
		\Delta_n(f, \gaa)  =  \sum_{i = 0}^n(f^i)^*(\gaa) = \sum_{i = 0}^n\sum_{j=0}^{k}{i\choose j} N^j\gaa = \sum_{j=0}^{k}\sum_{i = 0}^n{i\choose j} N^j\gaa.
	\end{equation}
	If $n \ge 0$, then 
	$$\sum_{i = 0}^n{i\choose j} = {n + 1\choose j+1}$$
	by the hockey-stick identity. Similarly,
	if $n \le 0$, then 
	$$\sum_{i = 0}^n{i\choose j} = (-1)^j\sum_{i = 0}^n{j-i - 1\choose j} = (-1)^j {j-n \choose j+1} =  
	- {n \choose j+1}.$$
	Hence Lemma~\ref{lem-gD+-} follows.
\end{proof}

The following lemma will be useful to prove results 
on lower bounds of $\Plov(f)$ in this paper.
In the projective setting, this lemma was due to Keeler~\cite[Lemma 6.5(4)]{Ke}
and was applied in his work to prove his lower bound.

\begin{lem}\label{lem-sdP}
	For every integer $i \in [1,d]$, let
	$$P_{f,\go,i}(n) \cnec \gD^+_n(f, \go)^i \go^{d-i} = \(\sum_{j=0}^{k}{n + 1\choose j+1} N^j\go\)^i\go^{d-i},$$
	which is a polynomial in $n$ of degree $\deg_{n} P_{f,\go,i}(n)$. 
	Then we have
	$$\deg_{n} P_{f,\go,i}(n) > \deg_{n} P_{f,\go,i-1}(n).$$
\end{lem}
\begin{proof}
	
	For every non-negative integer $m$, define 
	$$P_{f,\go,i,m}(n) \cnec \gD^+_n(f, \go)^{i-1}\cdot (f^m)^*\go \cdot \go^{d-i} = \(\sum_{j=0}^{k}{n + 1\choose j+1} N^j\go\)^{i-1} \cdot (\Id + N)^m(\go) \cdot \go^{d-i},$$
	which is a polynomial in $n$ of degree $\deg_{n} P_{f,\go,i,m}(n)$.
	Note that since both $\go$ and $(f^m)^*\go$ are K\"ahler, we have
	$$C_1 \go \le (f^m)^*\go \le C_2 \go$$
	for some $C_1,C_2 > 0$, so
	$$C_1 P_{f,\go,i,0}(n) \le P_{f,\go,i,m}(n) \le C_2 P_{f,\go,i,0}(n)$$
	by Lemma~\ref{lem:nef_compar} and therefore
	$$\deg_n P_{f,\go,i,m}(n) = \deg_n P_{f,\go,i,0}(n) = \deg_n P_{f,\go,i-1}(n).$$
	In particular, $\deg_n P_{f,\go,i,m}(n)$ is independent of $m$.
	
	For every $m$, 
	since $P_{f,\go,i,m}(n) > 0$, 
	the leading coefficient $C_{f,\go,i}(m)$ of the polynomial $P_{f,\go,i,m}$
	satisfies $C_{f,\go,i}(m) > 0$.
	As $C_{f,\go,i}(m)$ is a polynomial in $m$ (because $N$ is nilpotent),
	the minimum of 
	$$	\left\{C_{f,\go,i}(m) \mid m \in \Z_{\ge 0} \right\}$$
	exists;
	let $\ell \in \Z_{\ge 0}$ such that $C_{f,\go,i}(\ell)$ is the minimum.
	
	By construction,  we have
	$$\lim_{n \to \infty} \frac{P_{f,\go,i,m}(n)}{P_{f,\go,i,\ell}(n)} = \frac{C_{f,\go,i}(m)}{C_{f,\go,i}(\ell)} \ge 1$$
	for every $m \in \Z$. So
	$$\frac{P_{f,\go,i}(n)}{P_{f,\go,i,\ell}(n)} 
	= \sum_{m = 0}^n\frac{P_{f,\go,i,m}(n)}{P_{f,\go,i,\ell}(n)} \succeq_{n \to \infty} \gamma,$$
	for any $\gamma>0$,
	which shows that $\deg_n P_{f,\go,i,\ell}(n) < \deg_n P_{f,\go,i}(n)$.
	Hence
	$$\deg_n P_{f,\go,i-1}(n) = \deg_n P_{f,\go,i,\ell}(n) < \deg_n P_{f,\go,i}(n).$$
\end{proof}

\begin{lem} \label{lem35}
	$\Plov(f)$ is equal to the degree of the polynomial
	$$n \mapsto P_{f,\go}(n) \cnec P_{f,\go, d}(n) = \gD^+_n(f, \go)^d  =  \(\sum_{j=0}^{k(f)}{n + 1\choose j+1} N^j\go\)^d$$
	for any K\"ahler class $\go$ on $X$.
	As a consequence, $\Plov(f)$ is a positive integer satisfying
	\begin{equation}\label{eqn-PlovNil}
		k(f) + d \, \le \, \Plov(f) \, \le \, d +
		\max \left\{\sum_{j = 1}^d i_j  \ \bigg{|} \
		i_j \in \Z_{\ge 0}, (N^{i_1}\go)\cdots(N^{i_d}\go) \ne 0 \right\},
	\end{equation}
	where $d = \dim X$. Also, the limit superior
	defining $\Plov(f,\go)$
	in Lemma~\ref{lem-sminvVol} 
	for any nef and big class $\go$
	is actually a limit.
\end{lem}

\begin{proof} 
	The first claim and 
	the last statement about the limit superior are clear by Lemma~\ref{lem-gD+-} and the definition of $\Plov(f)$. 
	Then, the upper bound of $\Plov(f)$ is clear by the equation (\ref{eqn-abel}). 
	
	For the lower bound, by Lemma~\ref{lem-sdP} with the notations therein, we have
	$$\Plov(f) = \deg_{n} P_{f,\go,d}(n) > \deg_{n} P_{f,\go,d-1}(n) > \cdots > \deg_{n} P_{f,\go,1}(n).$$
	As
	$$N^{k(f)}\go = k(f)! \cdot \lim_{m \to \infty} \frac{(f^*)^m(\go)}{m^{k(f)}}$$
	is nef and nonzero by definition of $k(f)$, we have $(N^{k(f)}\go) \cdot \go^{d-1} \ne 0$.
	So $\deg_{n} P_{f,\go,1}(n) \ge k(f)+1$, which shows that
	$\Plov(f) \ge k(f) + d$.
\end{proof}

\begin{remark}
	
	Based on $\Plov(f,\go) = \Plov(f,(f^*)^i\go)$ for any integer $i$,
	the last statement in Lemma~\ref{lem35} regarding the limit superior 
	still holds if $f^* : H^{1,1}(X) \cto$ is quasi-unipotent.
	We do not know whether it continues to hold without the quasi-unipotence assumption.
\end{remark}

Now we can prove that the polynomial logarithmic volume growth is also compatible with product.

\begin{lem}\label{lem-pruduct}
	Let $X_i$ ($i=1$, $2$) be compact K\"ahler manifolds and let
	$f_i \in \Aut(X_i)$ 
	(without assuming that $f_i^* : H^{1,1}(X_i) \cto $ is unipotent).
	Then
	$$\Plov (f_1 \times f_2) = \Plov (f_1) + \Plov (f_2).$$
\end{lem}

\begin{proof}
	
	Let $\go_i$ be a K\"ahler metric on $X_i$ and
	let $\pr_i : X_1 \times X_2 \to X_i$ be the projection to the $i$-th factor.
	Then
	\begin{equation}\label{eqn-multVol}
		\Vol_{\pr_1^*\go_1 + \pr_2^*\go_2}(\Gamma_{f_1 \times f_2}(n)) =
		\Vol_{\go_1}(\Gamma_{f_1}(n))\Vol_{\go_2}(\Gamma_{f_2}(n)).
	\end{equation}
	which proves Lemma~\ref{lem-pruduct} in the case where $\Plov (f_1) = \infty$ or $\Plov (f_2) = \infty$.
	
	Assume that both $\Plov (f_1) $ and $\Plov (f_2)$ are finite,
	then $\Plov (f_1 \times f_2)$ is also finite by the equivalence $(1) \Leftrightarrow (6)$ in
	Lemma~\ref{lem:equiv_dyn} and the K\"unneth formula.
	To prove Lemma~\ref{lem-pruduct},
	by Lemma~\ref{lem-invfini} we can replace $f_1$ and $f_2$ by some common power. 
	Thus by Lemma~\ref{lem:equiv_dyn} again, we can assume that
	the actions of $f_1$, $f_2$, and $f_1 \times f_2$ acting on the cohomology rings 
	of $X_1$, $X_2$, and $X_1 \times X_2$ respectively are unipotent.
	It follows from Lemma~\ref{lem35} that the 
	limits superior in the definitions of $\Plov(f_1)$, $\Plov(f_2)$, and $\Plov(f_1 \times f_2)$ are all limits.
	Hence Lemma~\ref{lem-pruduct} for finite $\Plov (f_1) $ and $\Plov (f_2)$ 
	follows again from~\eqref{eqn-multVol}.
\end{proof}

\begin{cor}\label{cor-k=0}
	Let $X$ be a compact K\"ahler manifold of dimension $d$ and 
	let $f \in \Aut (X)$ be a zero entropy automorphism. Then
	$\Plov(f) = d$ if and only if $k(f) = 0$.
\end{cor}

\begin{proof}
	
	Since $f$ has zero entropy,
	by Lemmas~\ref{lem-invfini} and~\ref{lem:equiv_dyn}
	we can assume that $f^* : H^{1,1}(X)$ is unipotent.
	Thus $k(f) = 0$ implies $N = 0$, and  $\Plov(f) = d$ by
	Lemma~\ref{lem35}.
	Again by Lemma~\ref{lem35},
	$\Plov(f) = d$ implies $k(f) = 0$.
\end{proof}

Another consequence of Lemma~\ref{lem35} is the following.

\begin{cor}\label{cor-parity}
	$\Plov(f)$ has the same parity as $d = \dim X$.
\end{cor}

\begin{proof}
	Since $\Delta_n(f, \go)$ is K\"ahler, we have $\Delta_n(f, \go)^d > 0$ for all $n \in \Z$.
	So by Lemma~\ref{lem-gD+-}, we have
	\begin{equation}
		\begin{cases}
			\gD^+_{n}(f,\go)^d > 0 & \text{ for } n \gg 0 \\
			(-1)^d\gD^+_{n-1}(f,\go)^d > 0 & \text{ for } n \ll 0.
		\end{cases}
	\end{equation}
	It follows that the degree of the polynomial $n \mapsto \gD^+_{n}(f,\go)^d$, which is also $\Plov(f)$ by Lemma~\ref{lem35},
	has the same parity as $d$.
\end{proof}

The following lemma provides another way to compute $\Plov(f)$, and turns out to be useful.
Define 
$$ \Delta'_n(f, \go) \cnec   \sum_{i = 0}^n \((f^i)^* + (f^{-i})^*\)(\go).$$

\begin{lem}\label{lem-plovdbend}
	$\Plov(f)$ is also the degree of the polynomial
	$$n \mapsto \Delta'_n(f, \go)^d = \(\go + \sum_{i = -n}^n (f^i)^*\go\)^d.$$
\end{lem}

\begin{proof} 
	Recall that $\Plov(f)$ is defined as the polynomial degree of
	$n \mapsto \(\sum_{i = 0}^n (f^i)^*\go\)^d$, which is also the
	polynomial degree of
	$n \mapsto \(\sum_{i = 0}^{2n} (f^i)^*\go\)^d$ as well as $n \mapsto \(\sum_{i = 0}^{2n} 2(f^i)^*\go\)^d$. Hence
	Lemma~\ref{lem-plovdbend} follows from
	$$\(\go + \sum_{i = -n}^n (f^i)^*\go\)^d
	= \((f^{-n})^*\((f^{n})^*\go +\sum_{i = 0}^{2n} (f^i)^*\go\)\)^d
	=\((f^{n})^*\go +\sum_{i = 0}^{2n} (f^i)^*\go\)^d
	$$
	and
	$$\(\sum_{i=0}^{2n} (f^i)^*\omega\)^d \le \((f^n)^* \go + \sum_{i=0}^{2n} (f^i)^*\omega\)^d
	\le \(\sum_{i=0}
	^{2n} 2(f^i)^*\omega\)^d.$$
\end{proof}

As $f^* : H^{1,1}(X) \circlearrowleft$ is unipotent, $(f^{-1})^* : H^{1,1}(X) \circlearrowleft$ is also unipotent.
Set
$$ N' \cnec (f^{-1})^* - \Id :  H^{1,1}(X) \to H^{1,1}(X)$$
and let
$$N_m \cnec N^m + N'^m.$$
We have an analogous statement of Lemma~\ref{lem35} with $N^m$ replaced by $N_m$.

\begin{lem}\label{lem-ubN_i}
	We have
	$$\Plov(f) \le
	d + \max \left\{\sum_{j = 1}^d i_j  \ \bigg{|} \
	i_j \in \Z_{\ge 0}, (N_{i_1}\go)\cdots(N_{i_d}\go) \ne 0 \right\}.$$
\end{lem}

\begin{proof}
	Lemma~\ref{lem-ubN_i} follows from Lemma~\ref{lem-plovdbend} together with
	$$ \Delta'_n(f, \go)
	= \sum_{i = 0}^n (f^i)^*\go + \sum_{i = -n}^0 (f^i)^*\go
	= \sum_{i = 0}^{k(f)} \binom{n + 1}{i+1}(N^i + N'^i)(\go)
	= \sum_{i = 0}^{k(f)} \binom{n + 1}{i+1} N_i(\go).$$
\end{proof}

\begin{lem}\label{lem-sdPbis}
	For every integer $i \in [0,d]$, let
	$$P'_{f,\go,i}(n) \cnec \Delta'_n(f, \go)^i \go^{d-i} = \(\sum_{j=0}^{k}{n + 1\choose j+1} N_j\go\)^i\go^{d-i},$$
	which is a polynomial in $n$. 
	Then we have 
	$$\deg P'_{f,\go,i} > \deg P'_{f,\go,i-1}.$$
\end{lem}

\begin{proof}
	As in the proof of Lemma~\ref{lem-sdP}, for every non-negative integer $m$, we define
	$$P'_{f,\go,i,m}(n) \cnec \Delta'_n(f, \go)^{i-1}\cdot( (f^m)^*\go + (f^{-m})^*\go) \cdot \go^{d-i}.$$
	The same argument in Lemma~\ref{lem-sdP} shows that
	$$\deg_n P'_{f,\go,i,m}(n) = \deg_n P'_{f,\go,i,0}(n) = \deg_n P'_{f,\go,i-1}(n)$$
	for every $m$, and there exists $\ell \in \Z_{>0}$ such that
	the leading coefficient $C_{f,\go,i}(\ell) > 0$ of $P'_{f,\go,i,\ell}$ is minimum 
	among all $\ell \in \Z_{>0}$.
	Since $\Delta'_n(f, \go)^{i-1}\cdot \go \cdot \go^{d-i} > 0$ (because $\go$ is K\"ahler),
	it follows that
	$$\frac{P'_{f,\go,i}(n)}{P'_{f,\go,i,\ell}(n)} =  
	\sum_{m = 0}^n\frac{ P'_{f,\go,i,m}(n)}{P'_{f,\go,i,\ell}(n)} 
	\succeq_{n \to \infty} \gamma,$$
	for any $\gamma>0$, and we conclude the proof as in Lemma~\ref{lem-sdP}.
\end{proof}

Let $\go \in H^{1,1}(X)$. For all integer $0 \le p \le d$, consider the following
polynomial in $n$ with coefficients in $H^{p,p}(X)$:
$$Q_{f,\go,p} : n \mapsto  \((f^n)^*\go + (f^{-n})^*\go\)^p = \( \sum_{i = 0}^k \binom{n}{i}N_i(\go) \)^p.$$
{\it Let $\gl_p(f,\go)$ denote the
	polynomial degree of $Q_{f,\go,p}(n)$.}

\begin{remark}
	Note that for any product $\gO \in H^{d-p,d-p}(X)$
	of $d-p$ K\"ahler classes, $\gl_p(f,\go)$ is also the
	polynomial degree of
	$$n \mapsto  \gO \cdot \((f^n)^*\go + (f^{-n})^*\go\)^p.$$
	The same argument proving Lemma~\ref{lem-sminvVol}
	shows that the polynomial degree
	$\gl_p(f,\go)$ is independent of the choice of $\go$ whenever $\go$ is nef and big.
\end{remark}

We will use the next lemma in the proof of Lemma~\ref{lem-vank-2d4}.

\begin{lem}\label{lem-ubdpol} 
	For every integer $p$, we have
	\begin{equation*}
		\gl_p(f,\go) \le \max\,
		\left\{r_i \in \Z \ \ \big{|} \ \ \| (f^n)^* \cto H^{i,i}(X) \| \sim_{n \to \infty} C_in^{r_i}
		\text{ for some } C_i > 0, 0 \le i \le p \right\} .
	\end{equation*}
	In particular,
	$$\gl_p(f,\go) \le k(f) \left\lfloor \frac{d}{2}  \right\rfloor \le \frac{k(f)d}{2}  \le d(d-1).$$
\end{lem}
\begin{proof}
	The first statement follows from
	$$
	\((f^n)^*\go + (f^{-n})^*\go\)^p
	= \sum_{j=0}^p \binom{p}{j} (f^n)^*\go^{j} \cdot (f^{-n})^*\go^{p-j}
	= \sum_{j=0}^p \binom{p}{j} \((f^{2n})^*\go^{j} \)\cdot \go^{p-j}.
	$$
	For the last statement, the first inequality follows from~\cite[Proposition 5.8]{Dinh-bis}
	and the last inequality from Theorem~\ref{JBbound}.
\end{proof}

\section{Quasi-nef sequences and dynamical filtrations} \label{sec-filtr}

\ssec{Dynamical filtrations and proof of the upper bound~\eqref{ub-DLOZgk} in Theorem \ref{JBbound}}
\hfill

First we recall the definitions and basic properties of quasi-nef sequences
and dynamical filtrations.
We then prove some useful lemmas,
and finally the optimal upper bound~\eqref{ub-DLOZgk}  in Theorem \ref{JBbound}
(see Corollary~\ref{cor-JBboundkd}).

Let $X$ be a compact K\"ahler manifold of dimension $d \ge 1$.
For every $\gaa \in H^{i,i}(X,\R)$, if $\gaa \cdot H^{1,1}(X,\R)^{d-i} = 0$, we write
$$\gaa \equiv 0$$
as in Notations.
Let $\cK^i(X) \subset H^{i,i}(X,\R)$ be
the closed convex cone generated by classes of smooth positive $(i,i)$-forms.
We have $\cK^1(X) = \Nef(X)$, which is the nef cone of $X$.
For every $\gaa \in \cK^i(X)$, define
$$\Nef(\gaa) \cnec  \ol{\gaa \cdot \Nef(X)} \subset H^{i+1,i+1}(X,\R).$$
As $\Nef(X)$ is a convex cone, so is $\Nef(\gaa)$. Since $\Nef(\gaa) \subset \cK^{i+1}(X)$
and $\cK^{i+1}(X)$ is salient,
$\Nef(\gaa)$ is a closed salient cone.

\begin{construction}[Quasi-nef sequence~\cite{Zhang3}]
	
	Let  $f \in \Aut(X)$ be an automorphism of $X$ such that $f^* : H^{1,1}(X) \cto$ is unipotent.
	A \emph{quasi-nef sequence} (with respect to $f$) is a sequence
	$$M_1,\, \ldots, \, M_d \, \in \, H^{1,1}(X,\R)$$
	constructed recursively as follows.
	Suppose that $M_1,\ldots,M_i \in H^{1,1}(X,\R)$ are constructed,
	then $M_{i+1} \in H^{1,1}(X,\R)$ is an element such that
	\begin{itemize}
		\item $f^*(M_1\cdots M_{i+1}) = M_1\cdots M_{i+1} \ne 0$,
		\item $M_1\cdots M_i M_{i+1} \in \Nef(M_1\cdots M_{i})$.
	\end{itemize}
	Since $f^* : H^{1,1}(X) \cto$ is unipotent,
	the existence of $M_{i+1}$ follows from
	Birkhoff's Perron--Frobenius theorem~\cite{Bi} applied to $\Nef(M_1\cdots M_{i})$. 
	See also \cite[Theorem 1.1]{KOZ} for a generalization.
	We set $L_0 \cnec 1 \in H^0(X,\R)$ and define $L_i \cnec M_1\cdots M_i \in H^{i,i}(X,\R)$.
	
	Note that $M_1,\, \ldots, \, M_d \, \in \, H^{1,1}(X,\R)$ is also a quasi-nef sequence with respect to $f^{-1}$.  
	\qed
\end{construction}

Given a quasi-nef sequence $M_1,\ldots,M_d \in H^{1,1}(X,\R)$
with respect to an automorphism $f \in \Aut(X)$ such that $f^* : H^{1,1}(X) \cto$ is unipotent, define
$$F_i \cnec \big\{\gaa \in H^{1,1}(X,\R) \mid L_i\gaa \equiv 0 \big\}$$
and let $F'_i$ be the subspace of $F_i$ spanned by
$$\big\{\gaa \in F_i \mid L_{i-1} \gaa \equiv \gb \text{ for some } \gb \in \Nef(L_{i-1}) \big\}.$$
Recall from~\cite{DLOZ} that these subspaces form an $f^*$-stable filtration
\begin{equation}\label{filt-F'F}
	0 = F_0 \subset F'_1 \subset F_1 \subset \cdots \subset F'_{d-1} \subset F_{d-1} \subset F'_d = H^{1,1}(X,\R).
\end{equation}
We note that the filtration~\eqref{filt-F'F} 
\emph{depends on the choice} of 
a quasi-nef sequence $M_1,\ldots,M_d \in H^{1,1}(X,\R)$.
Here are some fundamental properties of these filtrations proven in~\cite{DLOZ}.

\begin{pro}[{\cite[Theorem 1.3]{DLOZ}}]\label{pro-filtr}
	\hfill
	\begin{enumerate}
		\item 
		We have $\dim (F'_i/F_{i-1}) \le 1$ and
		$$F'_i = \{ \gamma \in  F_i \mid L_{i-1}\gamma^2 \equiv 0\}.$$
		Moreover the following conditions are equivalent:
		\begin{itemize}
			\item[(i)]
			$F_{i-1} \ne F'_i$;
			\item[(ii)]
			$F'_i = F_{i-1} \oplus (\R \cdot M_i)$;
			\item[(iii)]
			$L_{i-1}M_i^2 = 0$.
		\end{itemize}
		
		\item There exist an integer $r \in [1,d-1]$
		and a strictly decreasing sequence of integers
		$$d-1 \ge s_1 > \cdots > s_r \ge 1$$
		such that for every K\"ahler class $\go \in H^{1,1}(X,\R)$ and every integer $j \in [1,r]$,
		$$(f^* - \Id)^{2j-1}\go \in F_{s_j} \bss F'_{s_j} \ \ \text{ and } \ \
		(f^* - \Id)^{2j}\go \in F'_{s_j} \bss F_{s_j - 1},$$
		and $(f^* - \Id)^{2r+1}\go = 0$. In particular, $(f^* - \Id)^{2r+1} = 0 \in \End(H^{1,1}(X,\R))$.
	\end{enumerate}
\end{pro}

The sequence $s_1 > \cdots > s_r$ in Proposition~\ref{pro-filtr} (2) 
depends on $f$ and is unique for a given 
quasi-nef sequence. The inverse $f^{-1}$ defines the same sequence
with respect to the same 
quasi-nef sequence by the next lemma.

\begin{lem}\label{lem-f-1filtr}
	Let $s_1 > \cdots > s_r$ be the sequence in Proposition~\ref{pro-filtr} (2)
	associated to $f$. Then
	for every K\"ahler class $\go \in H^{1,1}(X,\R)$ and every integer $j \in [1,r]$, we have
	$$((f^{-1})^* - \Id)^{2j-1}\go \in F_{s_j} \bss F'_{s_j} \ \ \text{ and } \ \
	((f^{-1})^* - \Id)^{2j}\go \in F'_{s_j} \bss F_{s_j - 1},$$
	and $((f^{-1})^* - \Id)^{2r+1}\go = 0$. 
\end{lem}

\begin{proof}
	Since both $F_{s_j}$ and $F_{s_j}'$ are $f^*$-invariant, we have
	$$((f^{-1})^* - \Id)^{2j-1}\go = (-1)^{2j-1}(f^{1-2j})^*(f^* - \Id)^{2j-1}\go \in F_{s_j} \bss F'_{s_j}.$$
	The same argument shows that $((f^{-1})^* - \Id)^{2j}\go \in F'_{s_j} \bss F_{s_j - 1}$
	and $((f^{-1})^* - \Id)^{2r+1}\go = 0$.
\end{proof}

The following two lemmas are both consequences of Proposition~\ref{pro-filtr} (1).

\begin{lem}\label{lem-vanprod}
	For $i \in [1,d] \cap \Z$,
	take $\eta_i \in F'_i$.
	Let $p \in [1,d] \cap \Z$ and $j \in [0,p] \cap \Z$. Then:
	\begin{enumerate}
		\item
		There exists some $C \in \R$
		such that
		$$L_j\eta_{j+1} \cdots \eta_p \equiv CL_p.$$
		\item For any $\eta \in F_p$, we have
		$$L_j\eta_{j+1} \cdots \eta_p\eta \equiv 0.$$
	\end{enumerate}
\end{lem}
\begin{proof}
	
	Since either $F'_p = F_{p-1}$ or
	$F'_p/F_{p-1}$ is a line spanned by $M_p + F_{p-1}$ by Proposition~\ref{pro-filtr} (1),
	there exists some $C_p \in \R$ such that
	$\eta_p - C_pM_p \in F_{p-1}$.
	As $L_{p-1}F_{p-1} \equiv 0$,
	we have
	$$L_{p-1}\eta_p \equiv C_p L_{p-1}M_p = C_p L_p.$$
	Induction proves that $L_j\eta_{j+1} \cdots \eta_p \equiv CL_p$ for some $C \in \R$.
	
	Since $L_p F_p \equiv 0$, (2) follows from (1) and the definition of $F_p$.
\end{proof}

\begin{lem}\label{lem-equalM}
	Assume that $M_1 =\cdots = M_i \in H^{1,1}(X,\R)$.
	Then
	$$F'_j = F_{j-1}$$
	for every $j \le i-1$.
\end{lem}

\begin{proof}
	Assume to the contrary that $F'_j \ne F_{j-1}$ for some $j \le i-1$.
	By Proposition~\ref{pro-filtr} (1), we would have
	$$L_{j+1} = L_{j-1}M_jM_{j+1} = L_{j-1}M_j^2  = 0,$$
	which is impossible.
	Hence $F'_j = F_{j-1}$ for every $j \le i-1$.
\end{proof}

As a consequence of these results,
we obtain the following refinements of
Theorem~\ref{JBbound}.

\begin{cor}\label{cor25K'}
	Let $\phi : X \to B$ be a surjective morphism with connected fibers
	between compact K\"ahler manifolds.
	Let $f \in \Aut (X)$ such that $f^* : H^{1,1}(X, \R) \cto$ is unipotent and
	$\phi^*\go_B$ is $f^*$-invariant for some K\"ahler class $\go_B \in H^{1,1}(B,\R)$. Then
	$$k(f) \le 2(\dim X - \dim B).$$
	Here, we recall that $k(f)+1$ is the maximal size of the Jordan blocks of the Jordan canonical form of the unipotent $f^*|_{H^{1,1}(X, \R)}$.
\end{cor}

\begin{proof}
	Let $m \cnec \dim B$.
	As $\phi^*\go_B$ is an $f^*$-invariant nef class and $\phi^*\go_B^m \not\equiv 0$, we can complete
	$$M_1 = \cdots = M_m = \phi^*\go_B$$
	to a quasi-nef sequence $M_1,\ldots,M_d$.
	By Lemma~\ref{lem-equalM}, we have
	$F'_j = F_{j-1}$
	for every $j \le m-1$.
	So according to Proposition~\ref{pro-filtr} (2) and the notation therein,
	necessarily $s_r \ge m$, so $r \le \dim X - \dim B$.
	Hence $(f^* - \Id)^{2(\dim X - \dim B) +1} (\go) = 0$ for every $\go \in H^{1,1}(X,\R)$.
\end{proof}

\begin{cor}\label{cor-JBboundkd}
	Let $X$ be a compact K\"ahler manifold of dimension $d \geq 1$ and of Kodaira dimension $\gk(X)$.
	Let $f \in \Aut(X)$ be an automorphism of zero entropy.
	\begin{enumerate}
		\item 
		We have
		$$k(f) \le 2(\dim X - \gk(X)).$$
		In other words,
		$$\|(f^{m})^*: H^{1,1}(X)\circlearrowleft\| =O\big(m^{2(d-\gk(X))}\big)$$
		as $m \to \infty$ for any norm of ${\rm End}_{\C}(H^{1,1}(X))$.
		\item The estimate in (1) is optimal, in the sense that
		for every $d \ge 1$ and $\gk \ge 1$, there exist some $X$ and $f \in \Aut(X)$ such that 
		$\dim(X) =  d$, $\gk(X) = \gk$, and
		$$\|(f^{m})^*: H^{1,1}(X)\circlearrowleft\| \sim_{m \to \infty} C m^{2(d-\gk(X))}$$
		for some $C > 0$.
	\end{enumerate}
\end{cor}

We prove first Corollary~\ref{cor-JBboundkd} (1).
We will prove  Corollary~\ref{cor-JBboundkd} (2) in Section~\ref{sec-MainTheorem2} 
by constructing explicit examples.

\begin{proof}[Proof of Corollary~\ref{cor-JBboundkd} (1)]
	By an equivariant K\"ahler desingularization, there exists
	a bimeromorphic morphism $\nu : \wt{X} \to X$ form a compact K\"ahler manifold $\wt{X}$
	such that $f$ lifts to an automorphism $\ti{f} \in \Aut(\wt{X})$ and that $\wt{X}$
	admits a surjective morphism $\phi : \wt{X} \to B$ to a projective manifold
	as a model of its Iitaka fibration.
	As $\phi$ is an Iitaka fibration of $\wt{X}$,
	$\ti{f}$ descends to a bimeromorphic self-map of $B$ of finite order
	by \cite[Theorem A]{NZ}.
	Up to replacing $f$ by a finite iteration of it,
	we can assume that $\phi$ is $\ti{f}$-invariant.
	In particular, $\phi^*\go_B$ is
	$\ti{f}^*$-invariant for every K\"ahler class $\go_B \in H^{1,1}(X,\R)$.
	
	Since $f$ has zero entropy,
	we have $d_1(\ti{f}) = 1$
	by 
	Lemma~\ref{lem:equiv_dyn} and~\cite[Theorem 1.1]{DNT} for the invariance under a generically finite map.
	Replacing $f$ by its finite iteration,
	we can assume that $\ti{f}^*: H^{1,1}(\wt{X},\C) \cto$ is unipotent by Lemma~\ref{lem:equiv_dyn}.
	Thus by Corollary~\ref{cor25K'}, we have
	$$\|(\ti{f}^{m})^*: H^{1,1}(\wt{X},\C)\circlearrowleft\|
	=_{m \to \infty} O\big(m^{2(d-\gk(X))}\big).$$
	As $H^{1,1}({X},\C) \hto H^{1,1}(\wt{X},\C)$ is
	$\ti{f}^*$-stable and the restriction of $\ti{f}^* : H^{1,1}(\wt{X},\C) \cto$
	to $H^{1,1}({X},\C)$ is ${f}^* : H^{1,1}({X},\C) \cto$, we have
	$$\|(f^{m})^*: H^{1,1}(X,\C)\circlearrowleft\| =_{m \to \infty}
	O\big(m^{2(d-\gk(X))}\big).$$
\end{proof}

\ssec{Some vanishing lemmas}
\hfill

{\it From now on till the end of Section~\ref{sec-filtr},
	$f \in \Aut(X)$ is an automorphism such that $f^*: H^{1,1}(X) \cto$ is unipotent.}
Under this assumption, $(f^{-1})^*: H^{1,1}(X) \cto$ is also unipotent.
Recall that in Section~\ref{sec-plov}, we have defined
$$N \cnec f^* - \Id \in \End(H^{1,1}(X,\R)) \ \ \text{ and } \ \
N' \cnec (f^{-1})^* - \Id \in \End(H^{1,1}(X,\R)),$$
and also $N_m \cnec N^m + N'^m$.

In this subsection, we will prove some vanishing results of intersections of $(1,1)$-classes 
which are images of $N_m$ or $N^m$. 
Let us start with
the following lemma.

\begin{lem}\label{lem-N2i-1}
	Let $\gaa \in H^{1,1}(X,\R)$.
	\begin{enumerate}
		\item Let $d-1 \ge s_1 > \cdots > s_r \ge 1$ be the sequence 
		associated to $f$ as in Proposition~\ref{pro-filtr} (2).  Then we have
		$$N_{2i - 1}(\go), N_{2i}(\go) \in F'_{s_i} \bss F_{s_i - 1}$$
		for any K\"ahler class $\go$.	
		In particular,
		$$N_{2i - 1}(\gaa), N_{2i}(\gaa) \in F_{d-i}'.$$
		\item
		Both $N_{k(f) - 1}(\gaa)$ and  $N_{k(f)}(\gaa)$ are $f^*$-invariant.
		\item
		If $\go$ is nef,
		then both $N_{k(f)}(\go)$ and $N_{k(f)-1}(\go)$ are nef.
	\end{enumerate}
\end{lem}

\begin{proof}
	First we prove (1). Note that $N + N' = -NN'$, so
	\begin{equation*}
		\begin{split}
			N_{2i-1}
			& = N^{2i - 1} +  N'^{2i - 1}
			= (N + N')\sum_{j=0}^{2i-2}(-1)^j N^jN'^{2i-2-j} \\
			& = \sum_{j=0}^{2i-2}(-1)^{j+1} N^{j+1}N'^{2i-1-j} \\
			& =  \sum_{j=0}^{2i-2} (-1)^{j+1} (f^* - \Id)^{j+1}\((f^{-1})^* - \Id\)^{2i-j-1} \\
			& =  \sum_{j=0}^{2i-2} (f^* - \Id)^{2i}(f^{-2i+j+1})^* =
			\sum_{j=0}^{2i-2} N^{2i} \circ (f^{-2i+j+1})^*
		\end{split}
	\end{equation*}
	Since  $\sum_{j=0}^{2i-2}(f^{-2i+j+1})^*\go$ is K\"ahler, Proposition~\ref{pro-filtr} (2) implies
	$$N_{2i-1}(\go) = N^{2i} \(\sum_{j=0}^{2i-2}(f^{-2i+j+1})^*\go\) \in F_{s_i}' \bss F_{s_i - 1}.$$ 
	
	By Proposition~\ref{pro-filtr} (2) and Lemma~\ref{lem-f-1filtr}, 
	we have $N^{2i}(\go), N'^{2i}(\go) \in F_{s_i}'$,
	so $N_{2i}(\go) \in F_{s_i}'$.
	Since $N^{2i}\go \in F_{s_i}'$,
	we have 
	$$L_{s_i - 1}N^{2i}\go \equiv C L_{s_i}$$
	for some $C \in \R$ by Lemma~\ref{lem-vanprod}. Since $N^{2i}\go \notin F_{s_i - 1}$, we have $C \ne 0$.
	Moreover, as $L_{s_i - 1}N^{p}\go \equiv 0$ for every $p > 2i$ by Proposition~\ref{pro-filtr} (2),
	we have
	$$L_{s_i - 1}N^{2i}\go = (2i)! \cdot \lim_{m \to \infty} L_{s_i - 1}\frac{(f^*)^m(\go)}{m^{2i}} \in \cK^{s_i}(X)/\equiv$$
	where $\cK^{s_i}(X)/\equiv$ denotes the image of $\cK^{s_i}(X)$ in $H^{s_i,s_i}(X,\R)/\equiv$.
	Since $L_{s_i} \in \cK^{s_i}(X)$, necessarily $C > 0$.
	Since $s_1 > \cdots > s_r$ is also the sequence 
	associated to $f^{-1}$ by Lemma~\ref{lem-f-1filtr},
	the same argument shows that there exists $C' > 0$ such that
	$$L_{s_i - 1}N'^{2i}\go \equiv C' L_{s_i}.$$
	Hence $$L_{s_i - 1}N_{2i}(\go) \equiv (C + C') L_{s_i} \not \equiv 0,$$
	namely $N_{2i}(\go) \notin F_{s_i - 1}$. 
	The last part follows from $F'_{s_i} \subset F'_{d-i}$, noting that $s_i \le d-i$.
	
	For (2), recall that $k(f)$ is an even number (Theorem~\ref{JBbound})
	so we can write $k(f) = 2i$.
	Since $N^{2i+1} = 0$ and $f^* = \Id + N$, we have
	$$N^{2i}/(2i)! = \lim_{m \to \infty} (f^*)^m/m^{2i}$$
	whose image is hence $f^*$-invariant.
	Since $f^*$ commutes with $N$, and
	$N' = -N (f^{-1})^*$, we have $N'^{2i} = N^{2i} (f^{-2i})^*$ whose image is hence $f^*$-invariant too. Thus the images of $N_{2i} = N^{2i} + N'^{2i}$ and $N_{2i-1} = \sum_{j=0}^{2i-2} \, N^{2i} \circ (f^{-2i+j+1})^*$ are also $f^*$-invariant.
	
	For (3),
	$$N^{2i}(\go)= (2i)! \, \lim_{m \to \infty} (f^*)^m(\go)/m^{2i}, \,\,\, N_{2i-1}(\go) = \sum_{j=0}^{2i-2} \, N^{2i} ((f^{-2i+j+1})^*(\go))$$
	are clearly all nef.
\end{proof}

\begin{cor}\label{cor-escalier} 
	Let $\go$ be a K\"ahler class. Assume that $k(f) = 2d-2$. Then for every integer $\ell \ge 2$, we have
	$$N_{i_1}(\go) \cdots N_{i_\ell}(\go) \equiv 0$$
	whenever 
	$$i_j \ge 2(d-j)-1 \text{ for all } j \le \ell -2, 
	\text{ and } i_{\ell - 1}, i_\ell \ge 2(d-\ell +1)-1.$$	
	Moreover, whenever
	$$i_j \in \{ 2(d-j), 2(d-j)-1\} \text{ for all } j,$$
	there exists some $C \in \R$ such that
	$$N_{i_1}(\go)\cdots N_{i_j}(\go) \equiv C N_{2d-2}(\go)\cdots N_{2(d-j)}(\go).$$
\end{cor}

\begin{proof}
	Corollary~\ref{cor-escalier} follows directly from
	Lemma~\ref{lem-vanprod} and Lemma~\ref{lem-N2i-1}. 
	Indeed, by the assumption and Lemma~\ref{lem-N2i-1}, we have 
	$$N_{i_1}(\go) \in F_1',\, \ \ \
	N_{i_2}(\go) \in F_2', \ \ \ \ldots, \ \ \ 
	N_{i_{l-2}}(\go) \in F_{l-2}', \ \ \ N_{i_{l-1}}(\go), N_{i_{l}}(\go) \in F_{l-1}'.$$ 
	Thus the first assertion follows from Lemma~\ref{lem-vanprod} (2).
	
	Similarly, by the assumption and Lemma~\ref{lem-N2i-1}, we have 
	$$N_{i_1}(\go), N_{2d-2}(\go) \in F_1',\, \ \ \ 
	N_{i_2}(\go), N_{2d-4}(\go) \in F_2',  \ \ \ \ldots, \ \ \
	N_{i_j}(\go), N_{2(d-j)}(\go)  \in F_{j}'.$$
	Thus the second assertion follows from Lemma~\ref{lem-vanprod} (1). 
\end{proof}

\begin{lem}\label{lem-equiv2N}
	Let $m$ be a positive integer and let 
	$$\gS \cnec \left\{ N^{k(f)}(\go), N_{k(f)}(\go), N_{k(f)-1}(\go) \mid \go \in H^{1,1}(X,\R) \text{ K\"ahler } \right\}.$$
	Then the following conditions are equivalent:
	\begin{enumerate}
		\item $M_1\cdots M_m \not\equiv 0$ for some $M_1,\ldots,M_m \in \gS$.
		\item $M_1\cdots M_m \not\equiv 0$ for every $M_1,\ldots,M_m \in \gS$. 
	\end{enumerate}
\end{lem}

\begin{proof}
	Fix a positive integer $m$.
	It suffices to prove that (1) implies (2).
	To this end, it suffices to prove that given
	$M_1,\ldots,M_m , M_m' \in \gS$,	
	$$M_1\cdots M_{m-1} M_m \not\equiv 0 \ \ \text{ implies } \ \
	M_1\cdots M_{m-1} M'_m \not\equiv 0.$$
	Then we can replace each factor of $M_1\cdots M_m$
	by any choice of $m$ elements $M'_1,\ldots,M'_m \in \gS$ one by one and 
	obtain $M'_1\cdots M'_m \not\equiv 0$.
	
	Since every element of $\gS$ is 
	nef and $f^*$-invariant by Lemma~\ref{lem-N2i-1}, the sequence
	$M_1, \ldots, M_{m-1}$
	can be completed to a quasi-nef sequence.
	Since $L_{m-1}M_m = M_1\cdots M_{m-1} M_m \not\equiv 0$ and $M_m' \in \gS$ by assumption,
	Proposition~\ref{pro-filtr} (2) and Lemmas~\ref{lem-f-1filtr} and~\ref{lem-N2i-1} (1) 
	then imply that
	$$M_1\cdots M_{m-1} M'_m = L_{m-1} M'_m 
	\not\equiv 0.$$
\end{proof}

As for when we have $(N^{k(f)}\go)^i = 0$, we have the following.

\begin{lem}\label{lem-DP}
	$(N^{k(f)}\go)^i = 0$
	whenever $2i > d$.
\end{lem}
\begin{proof}
	Let $j \in \Z_{\ge 0}$. 
	Since $\|(f^n)^* : H^{1,1}(X) \cto \| = O(n^{k(f)})$, we have
	$$\|(f^n)^* : H^{j,j}(X) \cto\| = O(n^{jk(f)})$$
	by~\cite[Proposition 5.8]{Dinh-bis}. 
	Suppose that $(N^{k(f)}\go) ^j \ne 0$,
	then
	$$\|(f^n)^* : H^{j,j}(X) \cto \| \sim Cn^{jk(f)}.$$
	As 
	$$\|(f^n)^* : H^{j,j}(X) \cto \| \sim \|(f^n)^* : H^{d-j,d-j}(X) \cto \|,$$
	necessarily $(N^{k(f)}\go)^i = 0$
	whenever $2i > d$.
\end{proof}

\begin{cor}\label{cor-vankk-1}
	
	Let $m$ be a non-negative integer such that $N^{k(f)}(\go_0)^m \equiv 0$
	(or equivalently $N_{k(f)}(\go_0)^m \equiv 0$ by Lemma~\ref{lem-equiv2N})
	for some K\"ahler class $\go_0$.
	Then for every  $\go \in H^{1,1}(X,\R)$, we have
	$$N_{k(f)}(\go)^iN_{k(f) -1}(\go)^j \equiv 0$$
	whenever $i + j \ge m$.
	
	As a consequence, for every  $\go \in H^{1,1}(X,\R)$
	and every pair of non-negative integers $i$ and $j$ such that $2i + 2j > \min(d, 2d - k(f))$, we have
	$$N_{k(f)}(\go)^iN_{k(f) -1}(\go)^j \equiv 0.$$
\end{cor}

\begin{proof}
	Since the vanishing $N_{k(f)}(\go)^iN_{k(f) -1}(\go)^j \equiv 0$
	is a Zariski closed condition for $\go \in H^{1,1}(X,\R)$
	and since the K\"ahler cone is Zariski dense in $H^{1,1}(X,\R)$, we can assume that
	$\go \in H^{1,1}(X,\R)$ is K\"ahler.
	Then the first statement follows from Lemma~\ref{lem-equiv2N}.
	
	Now we prove the second statement. Once again, we can assume that $\go$ is K\"ahler.
	Recall that $ k(f) = 2\ell$ is an even number (see \eg Theorem~\ref{JBbound}). 
	By the first statement, it suffices to show that
	$$N_{k(f)}(\go)^{d-\ell + 1} \equiv 0,$$
	as we already know that $N_{k(f)}(\go)^i \equiv 0$ if $2i > d$ by Lemma~\ref{lem-DP}.
	To this end, we can assume that $N_{k(f)}(\go)^{d-\ell } \not \equiv 0$
	and complete
	$$M_1 = \cdots = M_{d-\ell} \cnec N_{k(f)}(\go)$$
	to a quasi-nef sequence.
	Then Lemma~\ref{lem-N2i-1} implies that
	$$ N_{k(f)}(\go)= N_{2\ell}(\go) \in F'_{d-\ell} \subset F_{d-\ell}.$$
	Hence
	$$  N_{k(f)}(\go)^{d-\ell+1} = L_{d-\ell } N_{k(f)}(\go) \equiv 0.$$
\end{proof}

\section{Upper bounds of $\Plov(f)$:  beginning of the proof of Theorem~\ref{uniform}} \label{proofuniform}

Let us first prove Keeler's upper bound in Theorem~\ref{uniform} (1). 

\begin{pro}\label{pro-KUB}
	Let $X$ be a compact K\"ahler manifold of dimension $d$ and 
	let $f \in \Aut (X)$ be a zero entropy automorphism.
	Suppose that $k(f) > 0$. Then we have
	$$\Plov (f) \le k(f)(d-1) + d.$$	
\end{pro}
We will first provide a sketch of Keeler's original proof, 
then an alternative proof using Corollary~\ref{cor-vankk-1}. 

\begin{proof}[Keeler's proof of Proposition~\ref{pro-KUB}] Recall that $\Plov (f)$ is the degree of the polynomial $P_{f, \go}(n)$ which is the same as the polynomial $(\Delta_n(f, L)^d)$ in Theorem \ref{Keeler} (6) if we replace the ample class $L$ by the K\"ahler class $\go$.
	Therefore, by setting $D = \omega$ and $P = f^*$ in the proof of \cite[Lemma 6.13]{Ke}, the purely cohomological proof of \cite[Lemma 6.13]{Ke} works without any further change, which proves the result.
\end{proof}

\begin{proof}[Second proof of Proposition~\ref{pro-KUB}] 
	By Lemmas~\ref{lem-invfini} and~\ref{lem:equiv_dyn}, 
	we can assume that $f^* : H^{1,1}(X) \cto$ is unipotent. 	
	Let $\gaa \in H^{1,1}(X,\R)$ and let
	$$k(f) \ge i_1 \ge \cdots \ge i_d \ge 0$$
	be $d$ integers such that
	$$\sum_{j = 1}^d i_j > k(f)(d-1).$$
	Write the product $N_{i_1}(\gaa) \cdots N_{i_d}(\gaa)$
	in the form
	$$\Pi \cnec N_{k(f)}(\gaa)^a N_{k(f) - 1}(\gaa)^b N_{i_{a+b+1}}(\gaa) \cdots N_{i_{d}}(\gaa)$$
	with $i_{a+b+1} \le k(f)-2$. Then $2a + 2b > 2d - k(f)$ by the assumption.
	It follows from Corollary~\ref{cor-vankk-1} that $\Pi = 0$.
	Thus $\Plov(f) \le k(f)(d-1) + d$ by Lemma~\ref{lem-ubN_i}.
\end{proof}

The main result of this section is the following sharpened
upper bound of $\Plov(f)$.

\begin{theorem}\label{thm-uniform}
	Let $X$ be a compact K\"ahler manifold of dimension $d$ and 
	let $f \in \Aut (X)$ be a zero entropy automorphism.
	Assume that $d \ge 3$ and $k(f) > 0$. Then
	$$\Plov(f) \le  k(f)(d-1) + d - 2. $$
	When $k(f) = 2$, we have the optimal upper bound:
	$$\Plov(f) \le 
	\begin{cases}
		2d & \text{ if } d \text{ is even;} \\
		2d-1 & \text{ if } d \text{ is odd.}
	\end{cases}
	$$
\end{theorem}

The above inequality for $k(f) = 2$ was originally due to F. Hu with a different proof~\cite{FHu}.
We will prove Theorem~\ref{thm-uniform}
based on results in  Section~\ref{sec-filtr}
about the dynamical filtrations. 
Let us first prove Theorem~\ref{thm-uniform} when $k(f) = 2$. 

\begin{proof}[Proof of Theorem~\ref{thm-uniform} when $k(f) = 2$] 
	
	By Lemmas~\ref{lem-invfini} and~\ref{lem:equiv_dyn}, we can assume that $f^* : H^{1,1}(X) \cto$ is unipotent. 	
	Let $\go$ be a K\"ahler class and let $i$ be the largest integer such that $(N^2\go)^i \ne 0$.
	By Lemma~\ref{lem-DP}, we have $i \le \lfloor d/2 \rfloor$.
	Since $(N^2\go)^{i+1} = 0$, it follows from Corollary~\ref{cor-vankk-1} that
	$$(N_2\go)^a(N_1\go)^b \equiv 0$$
	whenever $a + b > i $. Hence by Lemma~\ref{lem-ubN_i},
	$$\Plov(f) \le d + 2i \le d + 2\lfloor d/2 \rfloor.$$
	
	For optimal examples,
	let $S$ be any compact K\"ahler surface and $f \in \Aut(S)$ any automorphism with $k(f) = 2$
	(see e.g.~\cite[\S 4.1]{DLOZ} for an example where $S$ is a torus).
	Then $\Plov(f) = 4$ by Corollary~\ref{cor-d=3}.
	If $d = 2m$, then $k(f^{\times m})  =2$ for $f^{\times m} \in \Aut(S^m)$ by Lemma~\ref{lem-kprod} and
	$\Plov(f^{\times m}) = 4m = 2d$ by Lemma~\ref{lem-pruduct}.
	If $d = 2m + 1$, then we consider $f^{\times m} \times \Id_C \in \Aut(S^m \times C)$ where $C$ is any smooth projective curve.
\end{proof}

The proof of Theorem~\ref{thm-uniform} when $k(f) > 2$ follows from a different argument.
In Lemmas~\ref{lem-vank-2d} and \ref{lem-vank-2d4} below, let $X$ be a compact K\"ahler manifold of dimension $d \ge 1$ and $f \in \Aut(X)$ an automorphism such that $f^* : H^{1,1}(X) \cto$ is unipotent.

\begin{lem}\label{lem-vank-2d}
	Assume $k(f) > 0$.
	Let $(a,b)$ be a pair of non negative integers such that $2a + 2b \ge 2d-k(f)$. 
	Let $i_1 \ge \cdots \ge i_{d'} \ge 0$ be $d'$ integers.
	When $2a + 2b = 2d-k(f)$ we assume 
	$$\sum_{j = 1}^{d'} i_j > (k(f)-4)d' + 2.$$
	Then
	$$N_{k(f)}(\gaa)^a N_{k(f) - 1}(\gaa)^b N_{i_1}(\gaa) \cdots N_{i_{d'}}(\gaa) \equiv 0$$
	for every $\gaa \in H^{1,1}(X,\R)$.
\end{lem}

\begin{proof}
	If $2a + 2b > 2d-k(f)$, then
	we already have
	$$N_{k(f)}(\gaa)^a N_{k(f) - 1}(\gaa)^b \equiv 0$$
	by Corollary~\ref{cor-vankk-1}.
	So we can assume that $2a + 2b = 2d-k(f)$, and 
	that $i_1 \le k(f) - 2$.
	
	We can also assume that
	$N_{k(f)}(\gaa)^a N_{k(f) - 1}(\gaa)^b  \not\equiv 0$ and that $\gaa \in H^{1,1}(X,\R)$ is K\"ahler.
	By Lemma \ref{lem-N2i-1}, both $N_{k(f)}(\gaa)$ and $ N_{k(f) - 1}(\gaa)$ are $f^*$-invariant nef,
	so we can complete
	$$M_1 = \cdots = M_a \cnec N_{k(f)}(\gaa)  , \ \ M_{a+1} = \cdots = M_{a+b} \cnec N_{k(f)-1}(\gaa)$$
	to a quasi-nef sequence.
	
	Since $\sum_{j = 1}^{d'} i_j > (k(f)-4)d' + 2,$
	we have
	$$i_1,i_2 \in [k(f)-3, k(f)-2].$$	
	Indeed,
	otherwise we would have $i_2 \le k(f) - 4$ and
	$\sum_{j = 1}^{d'} i_j \le (k(f) - 2) + (d'-1)(k(f) - 4)
	= (k(f)-4)d' + 2.$
	It follows from Lemma~\ref{lem-N2i-1} that
	$$N_{i_1}(\gaa), N_{i_{2}}(\gaa) \in  F'_{d-\frac{k(f)}{2} + 1 } = F'_{a + b + 1 }.$$
	So
	$$N_{k(f)}(\gaa)^a N_{k(f) - 1}(\gaa)^bN_{i_1}(\gaa) N_{i_{2}}(\gaa)
	= L_{a+b}N_{i_1}(\gaa) N_{i_{2}}(\gaa)  \equiv 0$$	
	by Lemma~\ref{lem-vanprod} (2), which proves Lemma~\ref{lem-vank-2d}.
\end{proof}

\begin{proof}[End of proof of Theorem~\ref{thm-uniform}]
	
	Recall that $k(f)$ is an even number (Theorem~\ref{JBbound}), and we already 
	proved the statement for $k(f) =2$. 
	It remains to prove Theorem~\ref{thm-uniform} for $k(f) \ge 4$.
	
	By Lemmas~\ref{lem-invfini} and~\ref{lem:equiv_dyn}, 
	we can assume that $f^* : H^{1,1}(X) \cto$ is unipotent. 	
	Let $\gaa \in H^{1,1}(X,\R)$ and let
	$$k(f) \ge i_1 \ge \cdots \ge i_d \ge 0$$
	be $d$ integers such that
	$$\sum_{j = 1}^d i_j > k(f)(d-1) - 2.$$
	Then the product $N_{i_1}(\gaa) \cdots N_{i_d}(\gaa)$
	is of the form
	$$\Pi \cnec N_{k(f)}(\gaa)^a N_{k(f) - 1}(\gaa)^b N_{i_{a+b+1}}(\gaa) \cdots N_{i_{d}}(\gaa)$$
	with $i_{a+b+1} \le k(f) - 2$.
	We now show that
	$\Pi = 0$.
	We have
	$$(a + b)k(f) + \sum_{j = a + b + 1}^di_j \ge  \sum_{j = 1}^{a + b}i_j + \sum_{j = a + b + 1}^di_j
	= \sum_{j = 1}^d i_j > k(f)(d-1) - 2.$$
	So if $d' \cnec d - a - b$, then
	$$ d'(k(f) - 2) \ge \sum_{j = a + b + 1}^di_j  > k(f)(d-a-b-1) - 2 = k(f)d' - k(f) - 2, $$
	which implies $k(f) + 2 > 2d'$. As $k(f)$ is even, we have $k(f) \ge 2d'$,
	namely $2a + 2b \ge 2d - k(f)$.
	Assume that $2a + 2b = 2d - k(f)$, namely $2d' = k(f)$, then
	since $2d' =  k(f) \ge 4$ by assumption, we have
	$$\sum_{j = a + b + 1}^di_j > k(f)d' - k(f) - 2 \ge (k(f)-4)d' + 2.$$
	It follows from Lemma~\ref{lem-vank-2d} that $\Pi = 0$,
	and thus Theorem~\ref{thm-uniform} follows from Lemma~\ref{lem-ubN_i}. 
\end{proof}

We finish this section by the following upper bound
of $\Plov(f)$ when $d \ge 4$, 
which improves the upper bound $\Plov(f) \le 2d^2-3d$ obtained by
combining Theorem~\ref{thm-uniform} and $k(f) \le 2d - 2$ in Theorem~\ref{JBbound}.

\begin{pro}\label{uniformdge4}
	Let $X$ be a compact K\"ahler manifold of dimension $d \ge 4$ and let $f \in \Aut (X)$ such that $d_1(f) =1$.
	Then
	$$
	\Plov(f) \le 2d^2-3d - 2.
	$$
\end{pro}

We first prove the following.

\begin{lem}\label{lem-vank-2d4}
	Assume that $k \cnec k(f) = 2d-2$ and $d \ge 4$. Take $d$ integers 
	$$k \ge i_1 \ge \cdots \ge i_d \ge 0$$
	such that
	$$\sum_{j = 1}^d i_j \ge  (k-2)d-1.$$
	Then for every $\gaa \in H^{1,1}(X)$, we have
	\begin{equation}\label{van-dge4}
		N_{i_1}(\gaa) \cdots N_{i_{d}}(\gaa) = 0.	
	\end{equation}
\end{lem}

\begin{proof}
	First we assume that $i_1 \ge k-1$. We have
	\begin{equation}
		\begin{split}
			\sum_{j = 2}^d i_j & = \(\sum_{j = 1}^d i_j\)  - i_1
			> (k-2)d-1 - k = 2d^2 -6d + 1 \\
			& \ge 2d^2 - 8d + 8 = (k - 4)(d-1) + 2,\end{split}
	\end{equation}
	where the second inequality follows from $d \ge 4$.
	So $N_{i_1}(\gaa) \cdots N_{i_{d}}(\gaa) = 0$ by Lemma~\ref{lem-vank-2d}.
	
	Assume that $i_1 \le k-2$.
	Since $\sum_{j = 1}^d i_j \ge (k-2)d - 1$ and the sequence $i_j$ is decreasing,
	necessarily
	$$i_1 = \cdots = i_{d-1} = k-2 \ \ \ \text{ and } \ \ \  i_d = k-2 \text{ or } k-3.$$
	Since we have already proven that $N_{j_1}(\gaa) \cdots N_{j_{d}}(\gaa) = 0$
	whenever $j_1 \ge k - 1$, in particular whenever
	$$\sum_{l = 1}^d j_l > (k-2)d,$$
	we have
	\begin{equation}\label{eqn-Pganmod}
		\begin{split}
			Q_{f,\gaa,d}(n)
			& \cnec \(\sum_{i = 0}^k \binom{n}{i}N_i(\gaa) \)^d \\
			& =_{n \to \infty} \binom{n}{k-2}^d N_{k-2}(\gaa)^d + d\binom{n}{k-2}^{d-1}\binom{n}{k-3} N_{k-2}(\gaa)^{d-1}N_{k-3}(\gaa) \\
			& + O(n^{(k-2)d - 2}).
		\end{split}
	\end{equation}
	Recall that  $\deg(Q_{f,\gaa,d}) \le d(d-1)$ by Lemma~\ref{lem-ubdpol}. 
	Since
	$d(d-1) \le (k-2)d - 2$ (because $d \ge 4$),
	it follows from~\eqref{eqn-Pganmod} that
	$N_{k-2}(\gaa)^d = 0$
	and then
	$N_{k-2}(\gaa)^{d-1}N_{k-3}(\gaa) = 0$,
	which proves Lemma~\ref{lem-vank-2d4}.
\end{proof}

\begin{proof}[Proof of Proposition~\ref{uniformdge4}]
	By Theorem~\ref{JBbound}, we have $k(f) \le 2d-2$ and $k(f)$ is even.
	Since $d \ge 4$, Proposition~\ref{uniformdge4}
	in the case $k(f) < 2d-2$ (resp. $k(f) = 2d-2$) follows from Theorem~\ref{uniform}
	(resp. Lemmas~\ref{lem-ubN_i} and \ref{lem-vank-2d4}).
\end{proof}

\section{A refined lower bound: end of the proof of Theorem~\ref{uniform} and Corollary~\ref{cor-kod}}\label{sec-lb}

In this section we prove the following
lower bound of $\Plov(f)$.
At the end we will conclude the proof of Theorem~\ref{uniform}
together with Corollary~\ref{cor-kod}.

\begin{thm}\label{thm-lb}
	Let $X$ be a compact K\"ahler manifold of dimension $d > 0$ and 
	let $f \in \Aut (X)$ be a zero entropy automorphism.
	Then we have 
	$$\Plov(f) \ge d + 2k(f) - 2.$$
\end{thm}

\begin{proof}
	
	We can assume that $\dim X \ge 2$, otherwise $k(f) = 0$, and Theorem~\ref{thm-lb} holds trivially.
	
	By Lemmas~\ref{lem-invfini} and~\ref{lem:equiv_dyn},
	we can assume that $f^* : H^{1,1}(X) \cto$ is unipotent. 
	Let $\go$ be a K\"ahler class. Recall that we have
	$$\gD_n'(f,\go) \cnec \sum_{i= 0}^n \((f^i)^*\go + (f^{-i})^*\go\)= \sum_{i = 0}^{k(f)} \binom{n + 1}{i+1} N_i(\go)$$
	from the computation in the proof of Lemma~\ref{lem-ubN_i}.
	By Lemma~\ref{lem-sdPbis}, and using the notations therein, we have
	$$\Plov(f) = \deg_{n} P'_{f,\go,d}(n) > \deg_{n} P'_{f,\go,d-1}(n) > \cdots > \deg_{n} P'_{f,\go,2}(n).$$
	Therefore it suffices to show that 
	$$\deg_{n} P'_{f,\go,2}(n) \ge 2k(f).$$
	
	Recall that
	\begin{equation}\label{eqn-P'2}
		P'_{f,\go,2}(n) = \Delta'_n(f, \go)^2 \go^{d-2} = \(\sum_{j=0}^{k(f)}{n + 1\choose j+1} N_j\go\)^2\go^{d-2}.
	\end{equation}
	Assume that $(N_{k(f)}\go)^2 \ne 0$. Since $N_{k(f)}\go$ 
	is nef by Lemma~\ref{lem-N2i-1} (3), we have $(N_{k(f)}\go)^2 \cdot \go^{d-2} \ne 0$. 
	Hence $\deg_{n} P'_{f,\go,2}(n) \ge 2k(f) + 2$ by~\eqref{eqn-P'2}.
	
	Now assume that $(N_{k(f)}\go)^2 = 0$. 
	Then
	\begin{equation}\label{van-k-1k}
		(N_{{k(f)}-1}\go)^2 \equiv 0, \ \ (N_{k(f)}\go)(N_{{k(f)}-1}\go) \equiv 0
	\end{equation}
	by Lemma~\ref{lem-equiv2N}.
	Since $N_{k(f)}\go$ is nef and $f^*$-invariant by Lemma~\ref{lem-N2i-1},
	we can construct a quasi-nef sequence $M_1,\ldots,M_d$ with $M_1 = N_{k(f)}\go$.
	Suppose that $(N_{k(f)}\go)(N_{{k(f)}-2}\go) \equiv 0$. Then $(N_{{k(f)}-2}\go) \in F_1$, 
	and we would have
	$(N_{{k(f)}}\go) \in F_0 = 0$ by Lemma~\ref{lem-N2i-1} (1), which contradicts the assumption that $N_{k(f)}\go \ne 0$. Hence $(N_{k(f)}\go)(N_{{k(f)}-2}\go) \not\equiv 0$. 
	Together with the vanishings~\eqref{van-k-1k} and~\eqref{eqn-P'2},
	we obtain $\deg_{n} P'_{f,\go,2}(n) = 2k(f)$.
\end{proof}

\begin{proof}[Proof of Theorem~\ref{uniform}]
	
	The upper bound and lower bound of $\Plov(f)$ in Theorem~\ref{uniform}
	follows from Theorems~\ref{thm-uniform} and \ref{thm-lb} respectively.
\end{proof}

\begin{proof}[Proof of Corollary~\ref{cor-kod}]
	The main statement of Corollary~\ref{cor-kod} follows from Theorem~\ref{uniform}
	(resp. Theorem~\ref{uniform}) when $k(f) > 0$ (resp. $k(f) = 0$).
	Together with Theorem~\ref{JBbound}, it follows that $\kappa (X) \ge d/2$ implies $\Plov (f) \le d^2-2$.
\end{proof}

\section{Complex tori: Proof of Theorem \ref{thm51} and a few remarks} \label{sec-MainTheorem3}

In this section, we prove Theorem~\ref{thm51}; 
see Remark~\ref{rem51} for further discussion.

\begin{proof}[Proof of Theorem~\ref{thm51}]
	
	First we perform some reduction.
	By Lemmas~\ref{lem-invfini} and~\ref{lem:equiv_dyn}, up to replacing $f$
	by some finite iteration of it, we can assume that $f^* : H^{1,0}(X)\cto$ is unipotent.
	Fix a basis
	$$dz_{1,1},dz_{1,2},\ldots,dz_{1,k_1},\ldots,dz_{p,1},\ldots,dz_{p,k_p}$$
	of $H^{1,0}(X)$ such that for every $i= 1,\ldots,p$,
	$$
	f^*dz_{i,j} =
	\begin{cases}
		dz_{i,1} \text{ if } j = 1 \\
		dz_{i,j} +  dz_{i,j-1} \text{ if } 2 \le j \le k_i.
	\end{cases}
	$$
	As the $f^*$-action on $H^{1,0}(X)$ 
	determines the $f^*$-action on $H^\bullet(X,\C)$ when $X$ is a torus (because $H^\bullet(X,\C)$ is generated by
	$H^1(X,\C) = H^{1,0}(X) \oplus \ol{H^{1,0}(X)}$), 
	by Corollary~\ref{cor-plovinv}
	we can assume that $X = E^d$ with $E$ being an elliptic curve, or even $E = \C/\Z[\sqrt{-1}]$, and $(z_{i,j})_{1 \le i \le p,  2 \le j \le k_i}$ are the global coordinates of $E^d$, so that
	$$f : E^d = \prod_{i = 1}^p E^{k_i} \to  \prod_{i = 1}^p E^{k_i} = E^d$$
	is the product of $E^{k_i} \to E^{k_i}$ defined by the unipotent Jordan matrix of size $k_i$.
	
	By the product formula (Proposition~\ref{prop1-Kahler} (4)), it suffices to prove Theorem~\ref{thm51} for the case $p = 1$. So, from now on until the end of proof, we assume that $p=1$. Namely $f^* : H^{1,0}(X)\circlearrowleft$ has only one Jordan block.
	
	Set $e_i = dz_{i}$ and $\bar{e}_i = d\bar{z}_{i}$.
	For every $\gs = \sum_{i,j} a_{ij}e_i \wedge \bar{e}_j
	\in H^{1,1}(X,\C) \setminus \{0\}$, define
	$$w(\gs) \cnec \max\{i+j \mid a_{ij} \ne 0\},$$
	and for every $p = 2,\ldots, 2d$, define
	$$\gs(p) \cnec \sum_{i + j = p} a_{ij}e_i \wedge \bar{e}_j.$$
	
	Note that $\sum_{i = 1}^d w_i (\gs_i) \le d(d+1)$ by definition. We need the following.	
	\begin{lem}\label{lem-van}
		Let $\gs_1,\ldots,\gs_d \in H^{1,1}(X,\C) \setminus \{0\}$ and let $w_i \cnec w(\gs_i)$.
		\begin{enumerate}
			\item If $\sum_{i = 1}^d w_i < d(d+1)$, then $\gs_1 \wedge \cdots \wedge \gs_d = 0$.
			\item If $\sum_{i = 1}^d w_i = d(d+1)$, then $\gs_1 \wedge \cdots \wedge \gs_d = \gs_1({w_1}) \wedge \cdots \wedge \gs_d({w_d})$.
		\end{enumerate}
	\end{lem}
	
	\begin{proof} By multi-linearity of $\gs_1 \wedge \cdots \wedge \gs_d$, it is clear that (1) implies (2), and that it suffices to prove (1) for $\gs_1,\ldots,\gs_d$ of the form $\gs_{i} = e_{i_1} \wedge \bar{e}_{j_1},\ldots, \gs_d = e_{i_d} \wedge \bar{e}_{j_d}$. If $\gs_1 \wedge \cdots \wedge \gs_d \ne 0$, then necessarily
		$$\{i_1,\ldots,i_d\} = \{1,\ldots,d\} = \{j_1,\ldots,j_d\},$$
		so $\sum_{i = 1}^d w_i = d(d+1)$.
	\end{proof}
	
	We return to the proof of Theorem \ref{thm51}.
	Let $N \cnec f^* - \Id$  and let
	$$\go \cnec \sqrt{-1} \sum_{i=1}^d e_i \wedge \bar{e}_i,$$
	which is a K\"ahler class on $X$.
	For every $q = 0,\ldots, 2d-2$, by induction on $q$ we have
	$$(N^q\go)(p) = 0$$
	for every $p > 2d-q$ and
	$$(N^q\go)(2d-q) = \sqrt{-1}(N^q(e_d \wedge \bar{e}_d))(2d-q) =
	\sqrt{-1}\sum_{i + j = q}\binom{q}{i}e_{d-q+i} \wedge \bar{e}_{d-q+j} \ne 0.$$
	Therefore,
	$$w(N^q\go) = 2d-q.$$
	Let $q_1,\ldots,q_d \ge 0$ be non-negative integers. If $\sum_{i=1}^d q_i > d^2 - d$, then by Lemma~\ref{lem-van}
	$$(N^{q_1}\go) \cdots (N^{q_d}\go) = 0,$$
	so $\Plov(f) \le d^2$ by~\eqref{eqn-PlovNil}.
	
	It remains to prove that $\Plov(f) \ge d^2$. Note that since $\go({2d}) = \sqrt{-1} e_d \wedge \bar{e}_d$ is nef,
	by Lemma~\ref{lem-nefle} we have
	$\Plov(f) \ge \Plov(f,\go({2d}))$.
	Until the end of the proof, we formally define $e_i \wedge \bar{e}_j = 0$
	whenever $i$ and $j$ are integers such that $i \notin [1,d]$ or $j \notin [1,d]$.
	
	\begin{claim}\label{claim-Pascalpyr}
		We have
		$$N^q(e_d \wedge \bar{e}_d) =
		\sum_{i + j \le q}\binom{q}{i,j,q-i-j}e_{d-q+i} \wedge \bar{e}_{d-q+j},$$
	\end{claim}
	\begin{proof}
		Let $V \cnec \C[X,Y]/(X^d,Y^d)$.
		We have an isomorphism of $\C$-vector spaces
		$V \simeq H^{1,1}(X)$ sending each
		$X^iY^j$ to $e_{d-i} \wedge {\bar e}_{d-j}$. Under this isomorphism, $N: H^{1,1}(X) \to H^{1,1}(X)$ becomes
		$$
		\begin{aligned}
			N : V & \to V \\
			P & \mapsto (XY + X + Y)P  \mod (X^d , Y^d),
		\end{aligned}
		$$
		so
		\begin{equation}\label{eqn-Nq(1)}
			N^q(1) = (XY + X + Y)^q = \sum_{i + j \le q}\binom{q}{i,j,q-i-j}X^{q-i}Y^{q-j} \mod (X^d , Y^d).
		\end{equation}
		Translating~\eqref{eqn-Nq(1)} back to $N: H^{1,1}(X) \to H^{1,1}(X)$
		proves the claim.
	\end{proof}
	
	For every integer $n > 0$,
	by Claim~\ref{claim-Pascalpyr} we have
	$$\gO \cnec
	\sum_{q=0}^{2d-2}{n\choose q+1} N^q(\go({2d}))
	= \sqrt{-1} \sum_{q=0}^{2d-2}  \sum_{i + j \le q}
	{n\choose q+1}\binom{q}{i,j,q-i-j}e_{d-q+i} \wedge \bar{e}_{d-q+j}.$$
	For each pair of integers $1 \le i,j \le d$,
	define the polynomial $P_{i,j}(n)$ in $n$ by
	\begin{equation}\label{eqn-defPij}
		\gO =
		\sqrt{-1} \sum_{1 \le i , j \le d}
		P_{i,j}(n) e_{i} \wedge \bar{e}_{j}.
	\end{equation}
	
	\begin{claim}\label{claim-deglc}
		The polynomial $P_{d-i,d-j}(n)$ in $n$ has degree
		$i + j + 1$ and leading coefficient
		$$\frac{1}{(i+j+1)!}
		\binom{i+j}{i}.$$
	\end{claim}
	
	\begin{proof}
		As $e_{d-i} \wedge \bar{e}_{d-j} = e_{d-q + (q-i)} \wedge \bar{e}_{d-q + (q-j)}$, 
		by construction we have (with $q$ varying in the sum)
		$$P_{d-i,d-j}(n) =
		\sum_{(q-i)+(q-j) \le q} {n\choose q+1}\binom{q}{q-i,q-j,i+j-q}.$$
		So the degree and the leading coefficient of $P_{d-i,d-j}$, are equal to 
		those of the polynomial ${n\choose q+1}\binom{q}{q-i,q-j,i+j-q}$ in $n$ when $q$ is maximal and satisfying $(q-i)+(q-j) \le q$ (that is, when $q = i+j$). This proves the claim.
	\end{proof}
	
	By~\eqref{eqn-defPij}, we have
	$$\gO^d = (\sqrt{-1})^d d! P(n)  (e_1 \wedge \bar{e}_1)  \wedge \cdots  \wedge (e_d \wedge \bar{e}_d)$$
	where $P(n)$ is the determinant of the matrix $(P_{i,j}(n))_{1 \le i,j\le d}$.
	By Claim~\ref{claim-deglc},
	we have $\deg_n(P(n)) \le d^2$ and the coefficient in front of $n^{d^2}$ is $\det M$,
	where $M = (M_{i+1, j+1})_{0 \le i,j \le d-1}$ is the $(d \times d)$ matrix defined by
	$$M_{i+1,j+1} = \frac{1}{(i+j+1)!} \binom{i+j}{i} = \frac{1}{i!j!} \cdot \frac{1}{(i+j+1)} , \,\,\, 0 \le i,j \le d-1.$$
	We have
	$$\det M = \frac{1}{\(\prod_{p=0}^{d-1}p!\)^2} \det\(\frac{1}{(i+j+1)}\)_{0 \le i,j\le d-1} =  \frac{\prod_{p=0}^{d-1}p!}{\prod_{p=d}^{2d-1}p!} \ne 0,$$
	where the second equality follows from
	the determinant of the Hilbert matrix (see \eg~\cite[(1.1)]{CauchyId}).
	Since $\Plov(f,\go({2d})) = \deg_n(P(n))$ by~\eqref{eqn-abel} and the definition of $\gO$,
	it thus follows that
	$$\Plov(f) \ge \Plov(f,\go({2d})) = \deg_n(P(n)) = d^2 .$$
	
	This completes the proof of the main statement of Theorem~\ref{thm51}.
	The optimality of the upper bound is provided by Example~\ref{ex-prod_E} below.
\end{proof}

\begin{example}\label{ex-prod_E}
	Let $E$ be a complex elliptic curve and let
	$X = E^d$. Define $f:  X  \to X$ by 
	\begin{equation}
		(x_1, \dots, x_d)  \mapsto (x_1, x_2+x_1, \dots, x_d+x_{d-1})
	\end{equation}
	Then $f^* : H^{1,0}(X)\cto$ is represented by the $(d \times d)$-Jordan matrix, and 
	$\Plov(X, f) = d^2$
	as a consequence of the main statement of Theorem~\ref{thm51}.
\end{example}

\begin{remark}\label{rem51}
	\hfill
	\begin{enumerate}	
		\item Consider $f \in \Aut (E^3)$ with $f(x_1,x_2, x_3) = (x_1, x_1+x_2, x_3)$.
		Then $f^* : H^{1,0}(E^3)\cto$ has two Jordan blocks, of sizes $2$ and $1$ respectively.
		By Theorem \ref{thm51}, we have
		$\Plov (f) =  2^2 + 1^2 = 5$,
		which is also consistent with \cite[Example 6.14]{Ke}.
		\item The upper bound in Theorem~\ref{thm51} 
		is also asserted in the proof of~\cite[Proposition 4.3]{CO} (without optimality). 
		However, the estimates (4.6)-(4.7) using $\ell_1$ in their proof have to be suitably modified, 
		otherwise
		as we can see that if $f$ is the identity, the estimate 
		$\Vol(\Gamma(n)) \le C n^{\ell_1^2}$ in~\cite[(4.6)-(4.7)]{CO}
		would imply that $\Vol(\Gamma(n))$ grows at most linearly in $n$,
		which contradicts the equality $\Plov(X) = d$.
	\end{enumerate}
\end{remark}

Finally, note that in this paper, whenever 
we prove that $\Plov(f)$ is bounded from above by some constant $C$
for an automorphism $f \in \Aut(X)$ such that $f^* : H^{1,1}(X) \cto$ is unipotent,
we actually prove that the right hand side of the inequality in 
Lemma~\ref{lem35} or \ref{lem-ubN_i} is bounded by $C$.
In view of Question~\ref{mainquest2}, we ask the following.

\begin{question}\label{que-mainquest2St}
	Let $X$ be a compact K\"ahler manifold of dimension $d \ge 1$ and 
	let $f \in \Aut (X)$ such that $f^* : H^{1,1}(X) \cto$ is unipotent.
	For every $\go \in H^{1,1}(X)$, do we have
	$$(N^{i_1}\go)\cdots(N^{i_d}\go) = 0$$
	and
	$$(N_{i_1}\go)\cdots(N_{i_d}\go) = 0$$
	whenever the $i_j$ are non-negative integers satisfying
	$\sum_{j = 1}^d i_j > d(d-1)$?  
\end{question}

\section{Some explicit examples} \label{sec-MainTheorem2}

We know that $\Plov (f) = 1$ for a compact Riemann surface, $\Plov (X, f) = 2$ or $4$ for a compact K\"ahler surface by Corollary~\ref{cor-d=3}, 
and $\Plov (X, f) = d$ for a projective variety $X$ of dimension $d$ whose desingularization $\tilde{X}$ is of general type (as $|{\rm Bir}\, (\tilde{X})| < \infty$ by \cite[Corollary 14.3]{Ueno}). 
Besides complex tori and these three cases,
we can also determine $\Plov (f)$ 
for some other classes of compact K\"ahler manifolds $X$ 
(Proposition~\ref{proHK}).
We  also prove Corollary~\ref{cor-JBboundkd} (2) in this section.

\begin{pro}\label{proHK}
	\hfill
	\begin{enumerate}
		\item Let $X$ be a compact hyper-K\"ahler manifold of dimension $2d$ and let $f \in \Aut (X)$ such that $d_1(f) = 1$. Then
		$\Plov (f) = 2d$ if $f$ is of finite order, and $\Plov (f) = 4d$ if $f$ is of infinite order; both cases are realizable, with $X$ projective. 
		\item Let $X$ be a smooth projective variety whose nef cone is a finite rational polyhedral cone. Let $\dim X = d$ and $f \in \Aut (X)$. Then $f$ is quasi-unipotent and $\Plov(f) = d$. In particular, this is the case when $X$ is a Mori dream space, especially when $X$ is a toric variety or a Fano manifold.
	\end{enumerate}
\end{pro}

\begin{proof}
	(1) The reader is referred to \cite{Hu} for basics about compact hyper-K\"ahler manifolds. Note that a compact hyper-K\"ahler manifold has no global vector field other than $0$. Hence $\Aut (X)$ is discrete. Thus $f$ is of finite order if and only if $f^* : {H^{1,1}(X, \R)} \cto$ is of finite order.
	
	So, replacing $f$ by its power and using Proposition~\ref{prop1-Kahler} (2), we can assume that $f = \id_X$ or $f^* : {H^{1,1}(X, \R)} \cto$ is unipotent of infinite order. The result is clear when $f = \id_X$. In the rest, we will assume that $f^* : {H^{1,1}(X, \R)} \cto$ is unipotent of infinite order.
	
	Let $q_X(x)$ be Beauville-Bogomolov-Fujiki's quadratic form on $H^{1,1}(X, \R)$. The signature of $q_X(x)$ is $(1, h^{1,1}(X) -1)$.
	
	Let $\go$ be a K\"ahler class on $X$. Then the degree of the polynomial
	$q_X(P_{f,\go}(n))$ in Lemma \ref{lem35},
	with respect to $n$ is $2^2 = 4$ by~\cite[Lemma 5.4]{AvB}. The first part of (1)
	then follows from Fujiki's relation below (with positive constant $c_X > 0$):
	$$(x^{2d})_X = c_X(q_X(x))^d .$$
	
	For the realization part of (1), let $\varphi : S \to \P^1$ be a projective elliptic K3 surface whose Mordell-Weil group ${\rm MW}\,(\varphi)$ has an element of infinite order, say $f$. There are plenty of such K3 surfaces. Then $f \in \Aut (S)$ and it induces an automorphism $f^{[d]} \in \Aut ({\rm Hilb}^d(S)/\P^d)$ of infinite order. Here the Hilbert scheme $X := {\rm Hilb}^d(S)$ is a projective hyper-K\"ahler manifold of dimension $2d$ with the Lagrangian fibration ${\rm Hilb}^d(S) \to \P^d$ induced by $\varphi$. Hence $d_1(f^{[d]}) = 1$ as it preserves the pullback $h$ of the hyperplane class of $\P^d$, which is non-zero nef class on $X$ such that $q_X(h) = 0$.
	Thus $(X, f^{[d]})$ provides an example such that $\Plov (f^{[d]}) = 2 \dim X = 4d$. This completes the proof of (1).
	
	(2) By the assumption, $f^* : N^1(X) \cto$ is always of finite order (even though the order of $f$ itself can be often infinite). Thus we have $\Plov(f) = d$ by Theorem~\ref{uniform}.
\end{proof}

We finish this section with proofs of 
Corollary~\ref{cor-JBboundkd} (2)
by constructing explicit examples. 
The examples that we will construct
also appear in other complex dynamical contexts~\cite{OT, DLOZ}.

\begin{proof}[Proof of Corollary~\ref{cor-JBboundkd} (2)]
	Let $X_d = E^d$ ($d \ge 2$) be the $d$-fold self-product of an elliptic curve $E$
	and $f_d$ the automorphism of $X_d$ defined by
	$$f_d(x_1, x_2, \ldots, x_d) = (x_1, x_2 + x_1, \ldots, x_{d} + x_{d-1}),$$
	as in Example \ref{ex-prod_E}. We have $k(f_d) = 2d-2$~\cite[\S 4.1]{DLOZ}. 
	
	Let $C$ be a smooth projective curve of genus $g(C) \ge 2$ with a surjective morphism
	$\pi : C \to E$.
	Let
	$Y_d := C \times E^{d-1}$.
	Then
	$Y_d$ is a smooth projective variety with $\dim Y_d = d$ 
	and Kodaira dimension $\kappa(Y_d) = 1$. 
	We define $g_d \in \Aut (Y_d)$ by
	$$g_d(P, x_2, x_3, \ldots, x_{d}) = (P, x_2+\pi(P), x_3+x_2 \ldots, x_{d}+x_{d-1})$$
	We also define
	$$p : Y_d \to X_d\,\, ;\,\, (P, x_2, x_3, \ldots, x_{d}) \mapsto (\pi(P), x_2, x_3, \ldots, x_{d}).$$
	Then $p$ is a finite surjective morphism such that
	$f_d \circ p = p \circ g_d$,
	so $k(g_d) = k(f_d) = 2d-2$ by Proposition~\ref{pro-birJB}.
	Finally, for every smooth projective variety $V$ with $\gk(V) = \dim V$, 
	let $V_d \cnec Y_d \times V$ 
	and consider $\phi_d \cnec g_d \times \Id_V \in \Aut(V_d)$.
	We have 
	$$2(d-1) = k(g_d) \le k(g_d \times \Id_V) \le 2(\dim V_d - \gk(V_d)) = 2(d-1),$$
	where the second inequality follows from the first statement of Corollary~\ref{cor-JBboundkd}.
	So 
	$$ k(\phi_d) = 2(\dim V_d - \gk(V_d)).$$
	When $d$ and $V$ vary, any pair of positive integers
	$\dim V_d \ge 1$ and $\gk(V_d) \ge 1$ is realizable, which finishes the proof.
\end{proof}

\section{Twisted homogeneous coordinate rings and GK-dimensions: 
	Proofs of Theorem~\ref{thm-comp} and Corollary~\ref{uniformGK}} \label{sec-THCR}

In this section, 
we first relate the polynomial log-volume growths $\Plov(f)$ 
to the GK-dimensions $\GK (X,f)$
of twisted homogeneous coordinate rings
through Keeler's work~\cite{Ke}.
Then we prove Theorem~\ref{thm-comp} and
Corollary~\ref{uniformGK},
explaining how the results of $\Plov(f)$ 
imply the analogous statements for $\GK (X,f)$.

\ssec{Recollection of Keeler's work~\cite{Ke}}
\hfill

Following~\cite{Ke}, we recall the definition of
twisted homogeneous coordinate rings and related notions,
together with the fundamental properties proven in~\cite{AvB} and~\cite{Ke}.

Let $X$ be an irreducible projective variety defined over an algebraically closed field $\bk$ of characteristic $0$.
Let $f \in \Aut(X)$ be an automorphism.
We say that a line bundle $L$ on $X$ is {\it $f$-ample} if for any coherent sheaf $F$ on $X$, there is a positive integer $m_F$ such that
$$H^q(X, F \otimes L \otimes f^*L \otimes \ldots \otimes (f^m)^*L) = 0$$
for any integer $q > 0$ and for any integer $m > m_F$.
A Cartier divisor $D$ is called $f$-ample if $\cO(D)$ is $f$-ample.

Let $f \in \Aut(X)$ 
and let $L$ be a line bundle on $X$. 
For any integer $n \in \Z_{\ge 0}$, define
$$\Delta_n(f, L) \cnec  L \otimes f^*L \otimes \ldots \otimes (f^n)^*L,$$
and
$$B_{n+1} (X, f, L)  \cnec H^0(X, \Delta_n(f, L)) ,  \ \ \ \ B_{0} = H^0(X, \cO_X) = \bk.$$
The {\it twisted homogeneous coordinate ring} of $X$ associated to $(f, L)$
is the (noncommutative) 
associative graded $\bk$-algebra
$$B(X, f, L) := \oplus_{n \in \Z_{\ge 0}} B_n(X, f, L).$$
The study of $B(X, f, L)$
was initiated by Artin and Van den Bergh~\cite{AvB}.
Together with the seminal work of Keeler~\cite{Ke},
here are some fundamental properties they proved.
In the statements, $N^1(X) \cnec \NS(X)/(\torsion)$.

\begin{theorem}[Keeler, Artin--Van den Bergh]\label{Keeler}
	Let $X$ be
	a projective variety of dimension $d>0$ and let $f \in \Aut (X)$
	\begin{enumerate}
		\item
		$f$-ample line bundles exist
		if and only if $f^* : N^1(X) \cto$ is quasi-unipotent.
	\end{enumerate}
	
	In the following, 
	we assume that $f^* : N^1(X) \cto$ is quasi-unipotent. 
	\begin{enumerate}
		\setcounter{enumi}{1}
		\item
		$L$ is $f$-ample if and only if there exists an integer $n > 0$ such that
		$\Delta_n(f, L)$ is ample.
		In particular, any ample line bundle is $f$-ample.
		\item The GK-dimension $\GK\, B$
		is independent of the choice of an $f$-ample line bundle $L$.
		We therefore define the GK-dimension as
		$$\GK (X, f) := \GK B(X, f, L)$$
		for any choice of $f$-ample line bundle $L$.
		\item	
		The GK-dimension is a positive integer.
		More precisely,
		after replacing $f$ by its suitable positive power so that $f^* : N^1(X) \cto$ is unipotent,
		the self intersection number $(\Delta_n(f, L)^d)$  is a polynomial in $n$ and its degree satisfies:
		$$\rGK (X, f) \cnec \GK (X, f) - 1 = \deg_n (\Delta_n(f, L)^d).$$
		We call $\rGK (X, f)$ the reduced GK-dimension of $(X,f)$. 
	\end{enumerate}
\end{theorem}

\begin{proof}
	Let us just indicate the references where these statements are proven.
	Statement (1) is contained in~\cite[Theorem 1.2]{Ke}.
	Statements (2), (3), and (4) follow from~\cite[Lemma 4.1]{AvB}, and~\cite[Proposition 6.11, Theorem 6.1.(1)]{Ke}, respectively.
\end{proof}

When $X$ is a complex projective manifold,
Keeler's work implies Theorem~\ref{thm-comp}
as an immediate corollary
that the reduced GK-dimension of $(X,f)$ coincides with
the polynomial log-volume growth of $f$.
Together with Theorem~\ref{Keeler}, this suggests
unexpected relations between noncommutative algebra and
complex dynamics of automorphisms of zero entropy.

\begin{proof}[Proof of Theorem~\ref{thm-comp}]
	Since $f^* : {N^1(X)} \cto$ is quasi-unipotent,
	by Theorem~\ref{Keeler} (4) and the definition of $\Plov(f)$,
	given any ample line bundle $L$ on $X$, we have
	$$\rGK B = \deg_n (\Delta_n(f, L)^d)
	=  \limsup_{n \to \infty} \frac{\log \Delta_n(f, c_1(L))^d}{\log n} = \Plov(f).$$
\end{proof}

\ssec{From K\"ahler to projective }\label{ssec-GKcor}
\hfill

Let $X$ be a projective variety over $\C$ and let $f \in \Aut(X)$.
Let $\nu : \wt{X} \to X$ be a projective desingularization of $X$ such that
$$f \circ \nu = \nu \circ \ti{f}$$ 
for some $\ti{f} \in \Aut(\wt{X})$ (see \eg~\cite[Theorem 3.45]{KollarResSing} for the existence).
Before we prove Corollary~\ref{uniformGK},
first we identify some dynamical properties and invariants of 
$(X,f)$ as a projective variety with those of $(\wt{X},\ti{f})$ as a compact K\"ahler manifold.

\begin{lem}\label{lem-equiv0ent}
	The following conditions are equivalent.
	\begin{enumerate}
		\item $f^* : {N^1(X)} \cto$ is quasi-unipotent.
		\item $\ti{f}^* : {N^1(\wt{X})} \cto$ is quasi-unipotent.
		\item $\ti{f}$ has zero entropy.
	\end{enumerate}
\end{lem} 

\begin{proof}
	We define $d_1(f)$ as in~\eqref{eqn-projddeg} but replacing $\go$ by an ample divisor. Then the same proof of~\cite[Proposition A.2 and Lemma A.7]{NZ} says that in the definition $d_1(f)$, we can assume that $\go$ is a nef and big divisor instead, and $d_1(f)$ is the spectral radius of $f^* : N^1(X) \cto$. Note that $d_1(f) = 1$ if and only if $f^* : {N^1(X)} \cto$ is quasi-unipotent
	by Kronecker's theorem.
	Then projection formula shows $d_1(f) = d_1(\ti{f})$, hence (1) and (2) are equivalent. 
	The equivalence between (2) and (3) follows from Lemma~\ref{lem:equiv_dyn},
	as we can compute $d_1(\ti{f})$ using ample classes (which also lie in $N^1(\wt{X})$).
\end{proof}

Assume that $f^* : {N^1(X)} \cto$ is quasi-unipotent. 
Then
$$\|(f^{m})^*: N^1(X) \otimes \C \circlearrowleft\| \sim_{m \to \infty} Cm^{k_\NS(f)}$$
for some $k_\NS(f) \in \Z_{\ge 0}$ and $C > 0$.

\begin{lem}\label{lem-kNS}
	We have
	$k_\NS(f) = k(\ti{f})$.
\end{lem}
\begin{proof}
	First of all, the same argument proving Proposition~\ref{pro-birJB}
	shows that
	$k_\NS(f) = k_\NS(\ti{f})$.
	It suffices to show that $k_\NS(\ti{f}) = k(\ti{f})$.
	
	Since $k(\ti{f})$ is invariant under finite iterations,
	by Lemma~\ref{lem-equiv0ent} and Lemma~\ref{lem:equiv_dyn}
	we can assume that $\ti{f}^* : H^{1,1}(\wt{X}) \cto$ is unipotent.
	As the ample cone of $\wt{X}$ spans $\NS(\wt{X})_\R$
	we can thus find an ample class $\go$ of $\wt{X}$ such that 
	$$(\ti{f}^* - \Id)^{k_\NS(\ti{f})}(\go) \ne 0 \ \ \ \text{ and } \ \ \  
	(\ti{f}^* - \Id)^{k_\NS(\ti{f}) + 1}(\go) = 0.$$
	Hence $k_\NS(\ti{f}) = k(\ti{f})$ by Proposition~\ref{pro-filtr} (2).
\end{proof}

\begin{lem}\label{lem-GKbir}
	We have
	$$\GK (X, f) = \GK B(X,f,L)$$
	for any big and nef line bundle $L$. As a consequence,
	$$\rGK (X, f) = \rGK (\wt{X}, \ti{f}) = \Plov(\ti{f})$$
	and $\GK (X, f)$ is a birational invariant.
\end{lem}

\begin{proof}
	Since $\GK (X, f)$ is the polynomial degree of $n \mapsto \deg_n (\Delta_n(f, L)^d)$
	by Theorem~\ref{Keeler} (4),
	the same argument of Lemma~\ref{lem-sminvVol} proves the first assertion of Lemma~\ref{lem-GKbir}.
	
	Let $L$ be an ample line bundle on $X$, since
	$$
	\(\sum_{i = 0}^n (\ti{f}^i)^*(\nu^*c_1(L))\)^d
	= \nu^*\(\sum_{i = 0}^n (f^i)^*c_1(L)\)^d
	= \(\sum_{i = 0}^n (f^i)^*c_1(L)\)^d,$$
	by Theorem~\ref{Keeler} (4) we have 
	$$ \GK B(\wt{X},\ti{f},\nu^*L) = \GK B(X,f,L) = \GK(X,f)$$
	As $\nu^*L$ is nef and big, it follows from the first statement
	that $\GK B(\wt{X},\ti{f},\nu^*L) = \GK(\wt{X},\ti{f})$,
	which finishes the proof of Lemma~\ref{lem-GKbir}.
\end{proof}

In the corollary below, $k(f)$ is defined with $f^* : H^{1,1}(X) \cto$ replaced by $f^* : N^1(X) \cto$ (denoted as $k_\NS(f)$ in Subsection~\ref{ssec-GKcor}).

\begin{cor}\label{uniformGK}
	Let $X$ be
	a projective variety of dimension $d > 0$ over $\bk$, not necessarily smooth,
	and let $f \in \Aut(X)$ be an automorphism such that $f^* : N^1(X) \cto$ is quasi-unipotent.
	Then the analogous statements of Theorem~\ref{uniform},  Corollary~\ref{cor-d=3}, 
	inequality~\eqref{ineq-dge4} (Proposition~\ref{uniformdge4}),
	etc.,
	hold
	with $\Plov(f)$ replaced by $ \rGK (X, f)$
	under the same assumptions on $d$ and $k(f)$. 
	In particular, we have
	$$ \rGK(X,f) \in \{3,5,9\}$$
	if $d = 3$ (by Corollary~\ref{cor-d=3}), and
	$$ \rGK(X,f) \le 2d^2-3d - 2$$
	whenever $d \ge 4$ (by Proposition~\ref{uniformdge4}).
\end{cor}

\begin{proof}
	By the Lefschetz principle, we can assume that the pair $(X,f)$ is defined over $\bk = \C$.
	Corollary~\ref{uniformGK} then follows the existence of 
	equivariant projective desingularization (see \eg~\cite[Theorem 3.45]{KollarResSing}),
	together with the comparison results Lemma~\ref{lem-GKbir} and Lemma~\ref{lem-kNS}.
\end{proof}

\section*{Acknowledgments}

	We thank Professor Tien-Cuong Dinh for the collaboration of our previous work \cite{DLOZ}, which is crucial in this work, and his warm encouragement. We also thank Professor Yoshinori Gongyo for the invitation to
	talk for one of us at the Tokyo-Kyoto Algebraic Geometry Zoom Seminar on 11th November 2020, which made our collaboration possible. We thank Professor Serge Cantat and Doctor Fei Hu for valuable comments,
	and Professor S. Paul Smith for his interest 
	and for bringing his
	related work to our attention.
	Finally, we thank the referee for the invaluable questions and comments.
	
	We are supported by the Ministry of Education Yushan Young Scholar Fellowship (NTU-110VV006),
	and the National Science and Technology Council
	(110-2628-M-002); the JSPS grant 20H00111, 20H01809; ARF: A-8000020-00-00, A-8002487-00-00 of NUS.

\end{document}